\documentclass[final, 12pt,times]{elsarticle}

\usepackage{amssymb}
\usepackage{amsthm}
\usepackage{amsmath}
\usepackage{amsopn}
\usepackage{amsfonts}

\usepackage{lipsum}
\usepackage{bbm}
\usepackage{graphicx}
\usepackage{epstopdf}
\usepackage{algorithmic}
\usepackage{caption}
\usepackage{xcolor}
\usepackage{subcaption}

\newtheorem{proposition}{Proposition}

\journal{Journal of Computational Physics}

\begin{document}

\begin{frontmatter}



\title{Numerical Schemes for 3-Wave Kinetic Equations: A Complete Treatment of the Collision Operator\tnoteref{t1}}
\tnotetext[t1]{The authors are  funded in part by  the  NSF RTG Grant DMS-1840260, NSF Grants DMS-1854453, DMS-2204795, DMS-2305523,    Humboldt Fellowship,   NSF CAREER  DMS-2044626/DMS-2303146, DMS-2306379. The first author gratefully acknowledges the support of the NNSA through the Laboratory Directed Research and Development
(LDRD) program at Los Alamos National Laboratory under project number 20220174ER. Los Alamos National Laboratory is
operated by Triad National Security, LLC for the U.S. Department of Energy’s NNSA.\\
LANL unlimited release number LA-UR-24-21681}
\author[1]{Steven Walton\corref{cor1}}
\ead{stevenw@lanl.gov}

\author[2]{Minh-Binh Tran}
\ead{minhbinh@tamu.edu}

\cortext[cor1]{Corresponding author}

\address[1]{{T-5, Theoretical Division, Los Alamos National Laboratory, Los Alamos, NM, USA}
}

\address[2]{{Texas A \& M University,
            College Station, TX, USA}}



\begin{abstract}
    In our previous work \cite{waltontranFVS}, numerical schemes for a simplified version of 3-wave kinetic equations, in which  only the simple forward-cascade terms of the collision operators are kept, have been successfully designed, especially to capture the long time dynamics of the equation given the multiple blow-up time phenomenon. In this second work in the series, we propose numerical treatments for the complete 3-wave kinetic equations, in which the complete, much more complicated  collision operators are fully considered  based on a  novel conservative form of the  equation.  We then derive an implicit finite volume scheme to solve  the  equation. The new discretization uses an adaptive time-stepping method which allows for the simulations to be carried to very long times.  Our computed solutions are compared with previously derived long-time asymptotic estimates for the decay rate of total energy of time-dependent solutions of 3-wave kinetic equations and found to be in excellent agreement.
\end{abstract}




\end{frontmatter}


\section{Introduction}
The theory of  wave turbulence has been  studied extensively in the physical literature, based on  the kinetic description of  weakly interacting waves (see, for instance,  \cite{Peierls:1993:BRK,Peierls:1960:QTS,hasselmann1962non,hasselmann1974spectral,benney1966nonlinear,zakharov2012kolmogorov,benney1969random} and the books \cite{Nazarenko:2011:WT,PomeauBinh,zakharov2012kolmogorov}).

Our current work is the second paper in our series of works that numerically study the so-called 3-wave kinetic equation (see  \cite{waltontranFVS} for the first part of our work)
\begin{equation}\label{WeakTurbulenceInitial}
	\begin{aligned}
		\partial_tf(t,p) \ =& \ \mathcal{Q}[f](t,p), \\\
		f(0,p) \ =& \ f_0(p).
	\end{aligned}
\end{equation}
The non-negative quantity  $f(t,p)$ is the wave density  at  wavenumber $p\in \mathbb{R}^N$, $N \ge 2$; $f_0(p)$ describes the initial condition.  The operator $\mathcal{Q}[f]$ is the integral collision operator, which reads
\begin{equation}\label{def-Qf}{\mathcal Q}[f](p) \ = \ \iint_{\mathbb{R}^{2N}} \Big[ R_{p,p_1,p_2}[f] - R_{p_1,p,p_2}[f] - R_{p_2,p,p_1}[f] \Big] \mathrm d ^Np_1\mathrm d ^Np_2 \end{equation}
with $$\begin{aligned}
	R_{p,p_1,p_2} [f]:=  |V_{p,p_1,p_2}|^2\delta(p-p_1-p_2)\delta(\omega -\omega_{1}-\omega_{2})(f_1f_2-ff_1-ff_2) 
\end{aligned}
$$
with the short-hand notation $f = f(t,p)$, $\omega = \omega(p)$ and $f_j = f(t,p_j),$ $\omega_j = \omega(p_j)$, for wavenumbers $p$, $p_j$, $j\in\{1,2\}$. The function $\omega(p)$ is the dispersion relation of the waves.
The form of the collision kernel $V_{p,p_1,p_2}$  depends on the type of wave system under consideration.

While most of the works in the physical literature focuses on  the existence of the so-called  Kolmogorov-Zakharov spectra, which is a class of {\it time-independent solutions}  of equation \eqref{WeakTurbulenceInitial}.
On the other hand, there are fewer works that study the  {\it time-dependent solutions} of 3-wave kinetic equations \eqref{WeakTurbulenceInitial}. The rigorous study of solutions to time dependent wave kinetic equations is now an important research direction, see, for instance  \cite{AlonsoGambaBinh,EscobedoBinh,GambaSmithBinh,nguyen2017quantum,ToanBinh,RumpfSofferTran,CraciunSmithBoldyrevBinh} for works concerning 3-wave kinetic equations, and  \cite{collot2024stability,EscobedoVelazquez:2015:FTB,EscobedoVelazquez:2015:OTT,germain2023local,germain2017optimal,SofferBinh1,staffilani2024energy,staffilani2024condensation} for works concerning 4-wave kinetic equation. {Among the very few works that numerically study   the time-dependent solutions of 3-wave kinetic equations, we would like to mention the series of remarkable works  \cite{connaughton2009numerical,connaughton2010aggregation,connaughton2010dynamical}. The goal of those works is to recover the ZK solutions, under several assumptions manually imposed on the evolution of the time-dependent solutions themselves.} In these works, the authors study the evolution equation \eqref{WeakTurbulenceInitial} with the following form of the 3-wave collision operator
\begin{align}\label{EE1Colm}
	\begin{split}
		\mathbb{{Q}}[f](t,\omega) \ =& \ \int_0^\infty\int_0^\infty \big[R(\omega, \omega_1, \omega_2)-R(\omega_1,\omega, \omega_2)-R(\omega_2, \omega_1, \omega) \big]{\mathrm d }\omega_1{\mathrm d }\omega_2, \\\
		& R(\omega, \omega_1, \omega_2):=  \delta (\omega-\omega_1-\omega_2)
		\left[ U(\omega_1,\omega_2)f_1f_2-U(\omega,\omega_1)ff_1-U(\omega,\omega_2)ff_2\right]\,,
	\end{split}
\end{align}
where   $| U(\omega_1,\omega_2) |  \ = \ (\omega_1\omega_2)^{\sigma/2}$.
In \cite{connaughton2010aggregation,connaughton2010dynamical}, the solutions are imposed to follow the   {\it dynamic scaling hypothesis}
$
	f(t,\omega)\approx s(t)^a F\left(\frac{\omega}{s(t)}\right),
$
whose energy is now
\begin{equation}\label{EnergyColm}
	\begin{aligned}
		\int_{0}^\infty \omega f(t,\omega)\mathrm d \omega\ = & \ \int_0^\infty s(t)^a F\left(\frac{\omega}{s(t)}\right)\omega \mathrm d \omega \ =  \  s(t)^{a+2} \int_0^\infty x F\left(x\right) \mathrm d x,
	\end{aligned}
\end{equation}
which has a growth $s(t)^{a+2}$.
{Using equation  \eqref{WeakTurbulenceInitial}-\eqref{EE1Colm}, } and equation for $F$ can also be computed
\begin{equation}
	\begin{aligned}
		& \dot{s}(t) \ =  \ s^\xi, \mbox{ with } \xi = \sigma + a + 2, \ \ 
		  aF(x) \ + \ x\dot{F}(x) \ =  \ \bar{\mathbb{Q}}[F](x).
	\end{aligned}
\end{equation}
From \eqref{EnergyColm}, the conservation of energy is obtained when  { $a=-2$.} By assuming  {\it the polynomial hypothesis} $F(x)\approx x^{-n}$, when $x\approx 0$, the quantity $n$ can be computed  $n=\sigma+1$. {  Hence,  the integral $\int_0^\infty xF(x) \mathrm dx$} becomes singular if $\sigma\ge 1$.\\
In \cite{connaughton2010dynamical},   $\sigma$ is considered in the interval $[0,2]$ and the solutions follow the so-called {\it forced turbulence hypothesis}  \begin{equation}\label{ColmEnergy2}
	\int_{0}^\infty \omega f(t,\omega)\mathrm d \omega=Jt,
\end{equation}
yielding $\dot{s} = \frac{J}{(a+2)\int_0^\infty xF(x)\mathrm d x}s^{-1-a}$ and then $a=-\frac{\sigma+3}{2}$.  Since $\int_0^\infty xF(x) dx$ with $F(x)\approx x^{{-\frac{\sigma+3}{2}}}$ for small $x$  diverges when {$\sigma\ge 1$}, many  approximations have to be imposed and the energy has the growth $s(t)^{a+2}$.

{It is therefore important that the evolution of time-dependent solutions of \eqref{WeakTurbulenceInitial} is studied numerically, under no further  assumptions on the solutions themselves. The  work  \cite{soffer2019energy} is among  the very first works that starts the study of {\it  time-dependent solutions, without further hypotheses}.  In \cite{soffer2019energy}, the case $\sigma=2$, which corresponds to capillary waves and acoustic waves, is considered, in which
\begin{equation}\label{EE1}
	|V_{p,p_1,p_2}|^2 \ = \ |p||p_1||p_2|, \ \ \ N=3  \mbox{ and } \omega(p)=|p| .\end{equation} 
In this work, a class of isotropic solutions $f(t,p)=f(t,|p|)$ to \eqref{WeakTurbulenceInitial}-\eqref{EE1} are studied, in which, the initial condition is radial 
$f_0(p)=f_0(|p|)$.
It is proved in \cite{soffer2019energy} that when no ad-hoc assumptions are imposed on the evolution of the solutions, the ZK solutions cannot be recovered. However, a different phenomenon can be observed, as we described below.
	\begin{itemize}
		\item[(i)]  The energy on the interval $[0,\infty)$ is  non-increasing in time and for all time $\tau_1>0$, there is a  larger time $\tau_2>\tau_1$ such that  
		$$\int_{[0,\infty)}f(\tau_2,|p|)\omega_{|p|} |p|^2\mathrm d \mu(p) \ < \ \int_{[0,\infty)}f(\tau_1,|p|)\omega_{|p|}|p|^2\mathrm d \mu(|p|),$$	 
		where $\mu$ denotes the measure of the extended line $[0,\infty]$ (see \cite{soffer2019energy}).
		\item[(ii)] For all $\epsilon\in(0,1)$, there exists $R_\epsilon>0$ such that
		$$\int_{\left[0, R_\epsilon \right)}f(t,|p|)\omega_{|p|}|p|^2\mathrm d \mu(|p|) \ =  \  \mathcal O(\epsilon), \ \ \   t\in[0,\infty).$$	
		\item[(iii)]  The energy cascade has an explicit rate
		$$\int_{\{|p|=\infty\}}f(t,|p|) \omega_{|p|}|p|^2\mathrm d \mu(|p|)\ \ge \ {C}_1 \ - \  \frac{{C}_2}{\sqrt{t}},$$
		where ${C}_1>0$ and ${C}_2>0$ are explicit constants.  The above inequality yields 	\begin{equation}\label{Decomposition0}\int_{0}^L f(t,|p|) \omega_{|p|}|p|^2\mathrm d \mu(|p|)\ \	\le \   \mathcal{O}\Bigg(\frac{1}{\sqrt{t}}\Bigg),\mbox{ for any $L>0.$}
		\end{equation}
        Note that the measure at infinity $\{|p|=\infty\}$ can be defined by using the  ``extended half real line'' $[0,\infty] \ = \ [0,\infty)\cup\{ \infty\}.$ 
         The measure $\mathrm d \mu(|p|)$ is the extension of the classical Lebesgue measure \cite{soffer2019energy}.
	We also deduce from the above inequalities the multiple blow-up time phenomenon: There exists a sequence of times
	\begin{equation} \label{Decomposition1} 0< t^*_1<t^*_2<\cdots<t_n^*<\cdots,\end{equation} 
	such that 
		$$0<\int_{\{|p|=\infty\}}f(t^*_1,|p|) \omega_{|p|}|p|^2\mathrm d \mu(|p|)$$ $$ < \int_{\{|p|=\infty\}}f(t^*_2,|p|)\omega_{|p|}|p|^2\mathrm d \mu(|p|)<\cdots.$$
	 	\end{itemize}
        It is the goal of our work to verify the discoveries of \cite{soffer2019energy} numerically: to develop numerical schemes that can  capture (i)-(ii)-(iii). Since (iii) implies (i)-(ii), we will only need to numerically test (iii).}

{{\it 	In our first work in the series \cite{waltontranFVS}, a simplified version of \eqref{WeakTurbulenceInitial}, in which  only the forward-cascade term of the collision operator is kept} (see \cite{connaughton2009numerical} for the definition of the forward-cascade term of the collision operator), has been studied. Even though only the forward-cascade term of the collision operator is kept, the phenomenon described in (iii) can also be observed numerically in \cite{waltontranFVS}.}
        
 {\it The success of the numerical observations of the much simpler case considered in \cite{waltontranFVS} motivates us to carry on the much more complicated treatment of the full collision operator in the current work.}   Different from the previous work \cite{waltontranFVS}, since in the current work, the full form of the collision operator is considered, we need to develop a complete conservative  form of  3-wave kinetic equations.  {\it  To our knowledge, our derived  conservative form of  3-wave kinetic equations has not appeared previously in the physical literature. This is one of the main novelties of the current work. Additionally, our choice of numerical discretization to solve this new conservative form has not appeared in the wave turbulence literature.}  
 
  We would also like to highlight the works \cite{krstulovic2025wavkins, zhu2024turbulence}, in which a different approach to discretizing the collision operator has been proposed. By performing the change of variables \( k = k_{\text{min}} \lambda^x \), for some constant \( k_{\text{min}} > 0 \) and \( \lambda > 1 \), one can discretize and truncate in the variable \( x \) rather than in \( k \). This implies that the corresponding discretization in \( k \) becomes coarser for large values of \( k \), leading to a nonuniform grid in \( k \). 
      In addition, in the important work \cite{bell2017self}, in which a self-similar profile of the solution for the Alfv\`en wave
 turbulence kinetic equation is studied. The solution is computed  before the first blow-up time $t_1^*$. Another important work \cite{semisalov2021numerical} studies a numerical method for solving the  self-similar profile before the first blow-up time $t_1^*$  for   4-wave kinetic equation using Chebyshev approximations.    {\it In contrast to those works, our main focus is the study of the multiple blow-up   phenomenon  \eqref{Decomposition1} as well as the bound \eqref{Decomposition0}, rather than the self-similar profile before the first blow-up time $t_1^*$.}

 \subsection*{Outline}
 The remainder of the article is structured as follows.  In the next section, section \ref{sec::conservative_form_derivation}, we give a derivation of the conservative form of the 3-wave kinetic equation.  This is followed in section \ref{sec::discretization} by the description of our implicit finite volume method with which we solve this new conservation law.  In section \ref{sec::num_tests}, the finite volume method is used to solve the evolution equation with various initial energy density distributions.  The obtained results are found to be in good agreement with the work \cite{soffer2019energy} already discussed in this introduction.  We conclude by summarizing our contributions and positing future directions in the numerical study of the time dynamic solutions of 3-wave kinetic equations.   
 
\section{Conservative Form of 3-WKEs}\label{sec::conservative_form_derivation}
This section presents the main result of the current work in Proposition \ref{prop::conservative_form}, in which we derive a conservative form of \eqref{WeakTurbulenceInitial}. {Note that  numerical schemes that approximate integrals on  unbounded domains require the truncations of the unbounded domains to  finite domains. The role of the conservative form is to reduce the number of terms in the collision operators, which reduces the number of truncations needed in the approximation. } However, rather than the wavenumber density, $f(t,p)$, which is not conserved in the case of 3-wave interactions, we will work with the energy density, $g(t,p) = pf(t,p)$, which is conserved for 3-wave systems.  Thus we have the following Cauchy problem for the energy density
\begin{equation}\label{eqn::energy_density_Cauchy}
    \begin{aligned}
        \partial_t g(t,p) = |p| \mathcal Q\Big[ \frac{g}{|p|}\Big](t,p)\\
        g(0,p) = g_0(p).
    \end{aligned} 
\end{equation}
{
Below, we recall Proposition 6 of \cite{soffer2019energy}, whose proof can be found in  \cite{soffer2019energy}.
\begin{proposition}\label{Lemma:WeakFormulation}
For radially symmetric functions $f(p):=f(|p|)$, $g(p):=g(|p|)$ and $\phi(p):=\phi(|p|)$, the following holds true
\begin{align}\label{WeakFormulation:radialL}
\begin{split}
& \  \int_{\mathbb{R}^3}\,\mathcal Q\left[\frac{g(p)}{|p|}\right]|p|\phi \mathrm d^3p   =  \ 16\pi^2\int_{|p_1|>|p_2|\ge0}{|p_1||p_2|}g(|p_1|)g(|p_2|)\times \\
 & \  \  \  \  \  \  \  \  \  \  \  \  \times \Big[||p_1|+|p_2||^3\phi(|p_1|+|p_2|)-2(|p_1|^2+|p_2|^2)|p_1|\phi(|p_1|)\\
 & \  \  \  \  \  \  \  \  \  \  \  \   -4|p_1||p_2|^2\phi(|p_2|)+(|p_1|-|p_2|)^3\phi(|p_1|-|p_2|)\Big]\mathrm d|p_{1}|\,\mathrm d|p_2|\,\\
 & \  \  \  \  \  \  \  \  \  \  \  \   + 8\pi^2\int_{|p_1|=|p_2|\ge0}{|p_1||p_2|}g(|p_1|)g(|p_2|)(|p_1|+|p_2|)^2\Big[(|p_1|+|p_2|)\phi(|p_1|+|p_2|)\\
 &\  \  \  \  \  \  \  \  \  \  \  \  -|p_1|\phi(|p_1|)-|p_2|\phi(|p_2|)\Big]\mathrm d|p_{1}|\,\mathrm d|p_2|\,.
\end{split}
\end{align}
In the rest of the paper, for the sake of simplicity, we omit the factor $8\pi^2$.
\end{proposition}}

We may now state the main result of the article.
\begin{proposition}\label{prop::conservative_form}
    The Cauchy problem for the energy density, equation \eqref{eqn::energy_density_Cauchy}, is equivalent to the conservation law, with $k = |p|$,
    \begin{equation}
        \begin{aligned}
            \partial_t g(t,k) -\partial_k \mathfrak Q[g](t,k) = 0 \\
            g(0,k) = g_0(k) ,
        \end{aligned}
    \end{equation}
    where $\mathfrak Q$ is a nonlinear, non-local flux, or collision-flux, of the energy density, given by
{ \begin{equation}\label{eqn::energy_density_flux}
    \begin{aligned}
        \mathfrak Q[g](t,k) = 2\Big[\int^\infty_0\int^{k+k_1}_k g_1 g_2(k_1-k_2)^2\mathrm d k_{2,1} - \int^k_0\int^{k_1}_{k-k_1}g_1g_2(k_1+k_2)^2\mathrm d k_{2,1} \\
        -4\int^k_0\int_k^\infty g_1 g_2 k_1 k_2 \mathrm dk_{2,1} - 4\int^k_0\int^{k_1}_0 g_1g_2k_1k_2\mathrm d k_{2,1}\\
        - 2\int^k_0 g^2_1 k_1^2 \mathrm dk_1 - 2\int_{\frac{k}{2}}^k g_1^2 k_1^2 \mathrm d k_1 \Big].
    \end{aligned}
    \end{equation}
}\end{proposition}
\begin{proof}
Following  \cite{soffer2019energy}, we apply a test function $\varphi(p)$ to equation \eqref{eqn::energy_density_Cauchy} and integrate to obtain
\begin{equation}
    \partial_t \int_{\mathbb R^3} g(t,p) \varphi(p) \mathrm d^3 p = \int_{\mathbb R^3} |p| \mathcal Q\Big[ \frac{g}{|p|}\Big](t,p) \varphi(p) \mathrm d^3 p.
\end{equation}
From \cite{soffer2019energy}, we should have $\varphi \in \mathfrak M$, where $\mathfrak M$ is the function space spanned by
\begin{equation}\label{eqn::space_M}
    \{ \varphi(p) \text{  }  | \text{  } p\varphi(p) \in C_c([0,\infty) )\},
\end{equation}
{where $C_c([0,\infty)) $ denotes the space of compactly supported continuous functions in $[0,\infty)$ } and the space $C([0,\infty])$ for which $\lim\limits_{{p\to \infty}}\varphi(p)$ exists.  Full details can be found in \cite{soffer2019energy}.

Assuming both $g(t,p)$ and $\varphi(p)$ are radial, we set $ k = |p|$ and we can identity $g(t,p)$ with $g(t,|p|)
$
and the same for $\varphi(k)$. We  define $\tilde g(t,k) = k g(t,k)$ and $\phi(k) = k\varphi(k)$. { From Proposition \ref{Lemma:WeakFormulation}, we have}
{
\begin{equation}
\begin{aligned}
        \int_{\mathbb R^3} |p| \mathcal Q\Big[ \frac{g}{|p|}\Big](t,p) \varphi(p) \mathrm d^3 p = 2\int_{|p_1|>|p_2|\geq 0}|p_1||p_2|g(|p_1|)g(|p_2|) \\
        \times \Big[||p_1|+|p_2||^3\varphi(|p_1|+|p_2|)-2(|p_1|^2 +| p_2|^2)|p_1|\varphi(|p_1|)\\ - 4p_1p_2^2\varphi(|p_2|) + (|p_1|-|p_2|)^3\varphi(|p_1|-|p_2|) \Big]\mathrm d |p_1|\mathrm d |p_2| \\
        + \int_{|p_1|=|p_2|\geq 0} |p_1||p_2|g(|p_1|)g(|p_2|)(|p_1|+|p_2|)^2\\
        \times\Big[ (|p_1|+|p_2|)\varphi(|p_1|+|p_2|)- |p_1|\varphi(|p_1|) \\
        - |p_2|\varphi(|p_2|) \Big]\mathrm d |p_1| \mathrm d |p_2|,
\end{aligned}
\end{equation}}
whence, using $\tilde g$ and $\phi$,
\begin{equation}
    \begin{aligned}
       = 2 \int_{k_1 > k_2 \geq 0} \tilde g(k_1) \tilde g(k_2)\Big[(k_1+k_2)^2\phi(k_1+k_2)-2(k_1^2+k_2^2)\phi(k_1) 
      \\ -4 k_1k_2\phi(k_2) + (k_1 - k_2)^2\phi(k_1-k_2)\Big] \mathrm d k_1 \mathrm d k_2\\
       + \int_{k_1 = k_2 \geq 0} \tilde g(k_1) \tilde g(k_2)(k_1 + k_2)^2 \Big[ \phi(k_1 + k_2) - \phi(k_1) - \phi(k_2)\Big]\mathrm d k_1 \mathrm d k_2,
    \end{aligned}
\end{equation}
which results in the following weak form of the $3$-WKE
\begin{equation}\label{3wke_weak}
    \begin{aligned}
        \partial_t \int_{\mathbb R^+} \tilde g(t,k)\phi(k)\mathrm d k = & & 
        \\ 2 \int_{k_1 > k_2 \geq 0} \tilde g(k_1) \tilde g(k_2)\Big[(k_1+k_2)^2\phi(k_1+k_2) 
        -2(k_1^2+k_2^2)\phi(k_1) 
       \\
       -4 k_1k_2\phi(k_2) + (k_1 - k_2)^2\phi(k_1-k_2)\Big] \mathrm d k_1 \mathrm d k_2
       \\ + \int_{k_1 = k_2 \geq 0} \tilde g(k_1) \tilde g(k_2)(k_1 + k_2)^2 \Big[ \phi(k_1 + k_2) - \phi(k_1) - \phi(k_2)\Big]\mathrm d k_1 \mathrm d k_2.
    \end{aligned}
\end{equation}
Following \cite{waltontranFVS}, we set $\phi(k) = \chi_{[0,c]} (k)$ which implies $\varphi(k) \in\mathfrak M$ as desired. 
{The left-hand side of \eqref{3wke_weak} becomes becomes $$ \partial_t \int_{\mathbb R^+}\tilde{g}(t,k)\chi_{[0,c]}(k)\mathrm{d}k = \partial_t \int^c_0 \tilde{g}(t,k)\mathrm{d}k, $$ which taking the derivative with respect to $c$ and using the definition of $\tilde{g}$ results in $$ \partial_c \int^c_0 \partial_t \tilde{g}(t,k)\mathrm{d} k = \partial_t \tilde{g}(t,c).$$ 
Now applying the derivative with respect to $c$ on the right-hand side,} we obtain 
\begin{equation}
    \begin{aligned}
        \partial_t \tilde g(t,c) =  
        \\ 2 \partial_c\Bigg[  \int_{k_1 > k_2 \geq 0} \tilde g(k_1) \tilde g(k_2) \Big[ |k_1 + k_2|^2 \chi_{[0,c]}(k_1+k_2) - 2 (k_1^2 + k_2^2)\chi_{[0,c]}(k_1)
        \\ - 4k_1k_2\chi_{[0,c]}(k_2) + (k_1 - k_2)^2\chi_{[0,c]}(k_1-k_2)\Big]\mathrm d k_1 \mathrm d k_2 \Bigg] 
        \\ + \partial_c\Bigg[  \int_{k_1 = k_2 \geq 0} \tilde g(k_1) \tilde g(k_2)(k_1 + k_2)^2 \Big[ \chi_{[0,c]}(k_1 + k_2)  -\chi_{[0,c]}(k_1) -\chi_{[0,c]}(k_2) \Big]\mathrm d k_1 \mathrm d k_2 \Bigg]
        \\
        =  2 \partial_c\int_{k_1 > k_2 \geq 0} \tilde g(k_1) \tilde g(k_2) \mathcal K_1 \mathrm d k_1 \mathrm d k_2 
       + \partial_c\int_{k_1 = k_2 \geq 0} \tilde g(k_1) \tilde g(k_2)\mathcal K_2 \mathrm d k_1 \mathrm d k_2 \\
       = \mathbb Q_1[g](t,c) + \mathbb Q_2[g](t,c),
    \end{aligned}
\end{equation}
where we have set
\begin{equation}\label{K_1}
    \mathcal K_1 = |k_1 + k_2|^2 \chi_{[0,c]}(k_1+k_2) - 2 (k_1^2 + k_2^2)\chi_{[0,c]}(k_1)
\end{equation}
$$ - 4k_1k_2\chi_{[0,c]}(k_2) + (k_1 - k_2)^2\chi_{[0,c]}(k_1-k_2),$$
and 
\begin{equation}\label{K_2}
    \mathcal K_2 = (k_1 + k_2)^2 \Big( \chi_{[0,c]}(k_1 + k_2)  -\chi_{[0,c]}(k_1) -\chi_{[0,c]}(k_2) \Big).
\end{equation}
{The effect of the test functions $\chi_{[0,c]}$ in the integrals above is to define the domain of integration in each term, as we will see shortly.  The above derivatives with respect to $c$ as written are defined in the weak sense.  We will see that once we have computed the terms $\mathcal{K}_1$, $\mathcal{K}_2$ and defined the bounds of integration, the derivatives are also well-defined in the classical sense and could be computed via an application of the Leibniz rule. }
As in \cite{waltontranFVS}, we can evaluate $\mathcal K_1 $ and $\mathcal K_2$ by considering sections of the $k_1-k_2$ domain which conform to the constraints imposed by the definitions of $\mathbb Q_1$ and $\mathbb Q_2$.  To this end, we start with $\mathcal K_1$. In each case we are constrained by the inequality $0 \leq k_2 < k_1$.
\begin{itemize}
    \item The case $k_1 + k_2 \leq c$ gives $\mathcal K_1 = -4k_1k_2$
    \item Next, consider $k_1 + k_2 > c$, then we have $5$ possibilities:
    \begin{itemize}
        \item First consider $k_1, k_2, k_1-k_2 \leq c$, then $\mathcal K_1 = -(k_1^2 + k_2^2 + 6k_1k_2)$
        \item Next, $k_1 > c$ with $k_2, k_1 - k_2 \leq c$ gives $\mathcal K_1 = k^2_1 + k^2_2 -6k_1k_2$.
        \item If $k_1, k_2 > c$ but $k_1 - k_2 \leq c$, we have $\mathcal K_1 = (k_1 - k_2)^2$.
        \item The case when $k_1, k_2, k_1-k_2 > c$ results in $\mathcal K_1 = 0$.
        \item Lastly, we can have $k_1, k_1 - k_2 > c$ and $k_2 \leq c$ which gives $K_1 = -4k_1k_2$.  
    \end{itemize}
\end{itemize}
Next, we analyze $\mathcal K_2$ which is constrained by $k_1= k_2$. 
\begin{itemize}
    \item First, $k_1 + k_2 \leq c$ or $k_1 \leq \frac{c}{2}$ which gives $\mathcal K_2 = -(k_1+k_2)^2$.
    \item If $k_1 + k_2 >c$ we have two cases
    \begin{itemize}
        \item The case $k_1 = k_2 > c$ which results in $\mathcal K_2 = 0.$
        \item Lastly, we can have $k_1 = k_2 \leq c$ giving $\mathcal K_2 = -2(k_1 + k_2)^2$.
    \end{itemize}
\end{itemize}
{Applying these results to $\mathbb Q_1$ and $\mathbb Q_2$ results in 
\begin{equation}\label{Q1_reduced}
    \begin{aligned}
        \mathbb{Q}_1 = 2\partial_c\Biggl(
        \underbrace{\int_{0}^{c}\int_{c}^{c+k_1}\tilde g_1 \tilde g_2\,(k_1^2+k_2^2-6k_1k_2)\,dk_{2,1}}_{\mathrm{(I)}}
        +\underbrace{\int_{c}^{\infty}\int_{c}^{c+k_1}\tilde g_1 \tilde g_2\,(k_1-k_2)^2\,dk_{2,1}}_{\mathrm{(II)}}\\
        -\underbrace{\int_{0}^{c}\int_{c-k_1}^{\,k_1}\tilde g_1 \tilde g_2\,(k_1^2+k_2^2+6k_1k_2)\,dk_{2,1}}_{\mathrm{(III)}}\\ 
        -4\,\underbrace{\int_{0}^{c}\int_{0}^{\,c-k_1}\tilde g_1 \tilde g_2\,k_1k_2\,dk_{2,1}}_{\mathrm{(IV)}}
        -4\,\underbrace{\int_{0}^{c}\int_{c+k_1}^{\infty}\tilde g_1 \tilde g_2\,k_1k_2\,dk_{2,1}}_{\mathrm{(V)}}
        \Biggr).
    \end{aligned}
\end{equation}
and
\begin{equation}\label{Q2_reduced}
    \begin{aligned}
        \mathbb Q_2 = -4\partial_c\Bigg(\underbrace{\int^{\frac{c}{2}}_0\tilde g_1^2 k_1^2 \mathrm d k_1}_{\mathrm{VI}} + 2\underbrace{\int^{c}_{\frac{c}{2}} \tilde g_1^2 k_1^2 \mathrm d k_1}_{\mathrm{VII}}\Bigg).
    \end{aligned}
\end{equation}
We used the notation $\tilde g_i = g(k_i)$ for $i=1,2$ and $\mathrm d k_{2,1} = \mathrm d k_2 \mathrm d k_1$ above and throughout the rest of the paper.  
Further simplifications can be made to reduce the number of terms.  First consider the term I, which we can write as
\begin{equation}\label{eqn:term1}
    \int_{0}^{c}\int_{c}^{c+k_1}\tilde g_1 \tilde g_2\,(k_1^2+k_2^2-6k_1k_2)\,dk_{2,1} = \int_{0}^{c}\int_{c}^{c+k_1}\tilde g_1 \tilde g_2\,(k_1-k_2)^2\,dk_{2,1}-4\int_{0}^{c}\int_{c}^{c+k_1}\tilde g_1 \tilde g_2\,k_1k_2\,dk_{2,1}.
\end{equation}
The first term of \eqref{eqn:term1} above can be combined with II of \eqref{Q1_reduced} to obtain
\begin{equation}
    \begin{aligned}
        \mathrm{I} + \mathrm{II} = \int_{0}^{c}\int_{c}^{c+k_1}\tilde g_1 \tilde g_2\,(k_1-k_2)^2\,dk_{2,1}\\-4\int_{0}^{c}\int_{c}^{c+k_1}\tilde g_1 \tilde g_2\,k_1k_2\,dk_{2,1} + \int_{c}^{\infty}\int_{c}^{c+k_1}\tilde g_1 \tilde g_2\,(k_1-k_2)^2\,dk_{2,1} \\
        = \int_0^\infty \int_c^{c+k_1}\tilde{g}_1\tilde{g}_2(k_1-k_2)^2\mathrm{d}k_{2,1} - 4\int_0^c\int_c^{c+k_1}\tilde{g}_1\tilde{g}_2k_1k_2\mathrm{d}k_{2,1},
    \end{aligned}
\end{equation}
further, combining the last term above with V gives
\begin{equation}
    \begin{aligned}
        \mathrm{I} + \mathrm{II} + \mathrm{V} = \int_0^\infty \int_c^{c+k_1}\tilde{g}_1\tilde{g}_2(k_1-k_2)^2\mathrm{d}k_{2,1} - 4\int_0^c\int_c^{\infty}\tilde{g}_1\tilde{g}_2k_1k_2\mathrm{d}k_{2,1}.
    \end{aligned}
\end{equation}
Similarly, term III in \eqref{Q1_reduced} can be written
\begin{equation}
    \begin{aligned}
        \mathrm{III} = -\int_{0}^{c}\int_{c-k_1}^{\,k_1}\tilde g_1 \tilde g_2\,(k_1+k_2)^2\,dk_{2,1}-4\int_{0}^{c}\int_{c-k_1}^{\,k_1}\tilde g_1 \tilde g_2\,k_1k_2\,dk_{2,1},
    \end{aligned}
\end{equation}
which combining with IV in \eqref{Q1_reduced} gives
\begin{equation}
    \begin{aligned}
        \mathrm{III} + \mathrm{IV} = -\int_{0}^{c}\int_{c-k_1}^{\,k_1}\tilde g_1 \tilde g_2\,(k_1+k_2)^2\,dk_{2,1}-4\int_{0}^{c}\int_{0}^{\,k_1}\tilde g_1 \tilde g_2\,k_1k_2\,dk_{2,1},
    \end{aligned}
\end{equation}
so that, applying these changes and rearranging terms in $\mathbb Q_1$, we get 
\begin{equation}
    \begin{aligned}
        \mathbb Q_1 = 2\partial_c\Bigg(\mathrm{I} + \mathrm{II} + \mathrm{V} + \mathrm{III} + \mathrm{IV} \Bigg)&\ \\
        = 2\partial_c\Bigg(\int_0^\infty \int_c^{c+k_1}\tilde{g}_1\tilde{g}_2(k_1-k_2)^2\mathrm{d}k_{2,1} - 4\int_0^c\int_c^{\infty}\tilde{g}_1\tilde{g}_2k_1k_2\mathrm{d}k_{2,1} \\
        -\int_{0}^{c}\int_{c-k_1}^{\,k_1}\tilde g_1 \tilde g_2\,(k_1+k_2)^2\,dk_{2,1}-4\int_{0}^{c}\int_{0}^{\,k_1}\tilde g_1 \tilde g_2\,k_1k_2\,dk_{2,1}\Bigg).
    \end{aligned}
\end{equation}
In $\mathbb Q_2 $, we can rearrange terms VI and VII slightly
\begin{equation}
    \begin{aligned}
        \mathrm{VI} + \mathrm{VII} = -4\partial_c\Bigg(\int_0^{\frac{c}{2}}\tilde{g}_1^2k_1^2 + \int_{\frac{c}{2}}^c\tilde{g}_1^2k_1^2 + \int_{\frac{c}{2}}^c\tilde{g}_1^2k_1^2\Bigg)\\
        = -4\partial_c\Bigg(\int_0^{c}\tilde{g}_1^2k_1^2 + \int_{\frac{c}{2}}^c\tilde{g}_1^2k_1^2\Bigg).
    \end{aligned}
\end{equation}
Dropping the tilde notation in what follows throughout, we arrive at
\begin{equation}\label{3wke_final}
    \begin{aligned}
        \partial_t g(t,c) = 2\partial_c\Bigg(\mathrm{I} + \mathrm{II} + \mathrm{V} + \mathrm{III} + \mathrm{IV} + \frac12\mathrm{VI} + \frac12\mathrm{VII}\Bigg) \\ 2\partial_c\Bigg(\int^\infty_0\int^{c+k_1}_c g_1 g_2(k_1-k_2)^2\mathrm d k_{2,1} 
        -4\int^c_0\int_c^\infty g_1 g_2 k_1 k_2 \mathrm dk_{2,1}\\ - \int^c_0\int^{k_1}_{c-k_1}g_1g_2(k_1+k_2)^2\mathrm d k_{2,1} - 4\int^c_0\int^{k_1}_0 g_1g_2k_1k_2\mathrm d k_{2,1}\\
        - 2\int^c_0 g^2_1 k_1^2 \mathrm dk_1 - 2\int_{\frac{c}{2}}^c g_1^2 k_1^2 \mathrm d k_1 \Bigg)\\
        = \partial_c \mathfrak Q[g](t,c),
    \end{aligned}
\end{equation}
which, after changing variables from $c$ to $k$, is the conservative form we wished to obtain.}
\end{proof}
  
This is the identity from which we will derive our numerical method in the next section.  We remark that the above form \eqref{3wke_final} will only require one to truncate two integrals in the numerical implementation, specifically the semi-infinite integrals in the first and third terms on the right-hand side above.

\subsection*{Remark}

{Choosing the test function $\chi_{[0,c]}$ in the above conservative form is an important step in designing our numerical scheme. In the  previous work \cite{waltontranFVS},  we only consider the forward cascade part of the collision operator. By using the test function  $\chi_{[0,c]}$, we reduce significantly the number of convolution integrals appearing in the forward cascade part of 
 collision operator. Since the collision operator involves convolution integrals from $0$ to $\infty$, truncations are needed. However, by reducing the number of convolutions integrals, the number of truncations needed for the case considered in \cite{waltontranFVS} is only $1$. In the current manuscript, when the backward cascade part is added to the collision operator, the number of truncations needed is only $2$. Since  there is no cancellation between the backward cascade and forward cascade parts of the collision operator, reducing the number of truncations to $2$ is an important step to enhance the accuracy of the numerical schemes.}  

In our numerical tests, we are interested in measuring the total energy decay, which is easy to obtain from $\tilde g(t,k) = g(t,k) k$. However, one can easily modify the computations above to obtain an energy flux in terms of $g(t,k)$ instead of $\tilde g(t,k)$.  In fact, one can keep our choice to use $\phi = k\varphi(k)$ above and leave $g(t,k)k $ in place of $\tilde g(t,k)$. The computations for $\mathcal K_1$ and $\mathcal K_2$ are identical, and thus the resulting domains of integration remain as above.

In the next section, we give a full account of the implicit finite volume discretization used to solve \eqref{eqn::energy_density_Cauchy}.

\section{Discretized Equations}\label{sec::discretization}
We now describe our discretization of the 3-wave kinetic equation 
\begin{equation}\label{final_cauchy}
    \begin{aligned}
            \partial_t g(t,k) -\partial_k \mathfrak Q[g](t,k) = 0 && (t,k) \in \mathbb R^+\times \mathbb R^+, \\
            g(0,k) = g_0(k) && (t,k) \in \{0\}\times\mathbb R^+.
    \end{aligned}
\end{equation}
   The first step is to truncate the computational domain so that $(t,k) \in I\times\Omega \subset \mathbb R^+\times\mathbb R^+ $, where $I = (0,T]$, $T>0$ and $\Omega = [0,L]$ with $L>0$ which will be referred to as the truncation parameter.  The time domain is partitioned into $N$ intervals, so that $t^{n+1} = t^{n} + \Delta t_n$ with $n\in\{0, 1, 2, \ldots, N\}$ as is standard.  The wavenumber domain is partitioned into a collection of cells, $K_i$, such that $K_i\cap K_j =\emptyset$ for $i\neq j$ and $\bigcup\limits_{K_i\in\mathcal K }K_i = \Omega$, with the collection of cell $\mathcal K=\{ K_i \}_{i=0}^M$.  The $K_i$ are defined to be $K_i = [k_{i-1/2}, k_{i+1/2} )$ and $\Delta k_i = k_{i+1/2} - k_{i-1/2}$ is the stepsize in the wavenumber domain with midpoints $k_i = \frac{k_{i+1/2}+k_{i-{1/2}}}{2}$ as usual.  

Let $g^n_i = g(t^n, k_i)$ be the grid function evaluated at the midpoint $k_i$ at time $t^n$ and $g^0_i = \frac{1}{\Delta k_i}\int_{K_i} g(t^0, k) \mathrm d k$ is the finite volume approximation of the initial condition.  

The discrete collision operator $\mathfrak Q_{i + 1/2}$ is defined by
\begin{equation}\label{discrete_collision_op}
    \begin{aligned}
        \mathfrak Q_{i + 1/2}g^n = 2\Bigg(\sum_{j=0}^M \Delta k_j g^n_j \sum_{\ell = i}^{\beta} \Delta k_\ell g^n_\ell (k_j - k_\ell)^2 - \sum_{j=0}^{i}\Delta k_j g^n_j \sum_{\ell = \sigma}^j\Delta k_\ell g^n_\ell (k_j + k_\ell)^2
        \\ 
        - 4\sum_{j=0}^{i}\Delta k_jg^n_j\sum_{\ell = i}^M\Delta k_\ell g^n_\ell k_j k_\ell
        - 4\sum_{j=0}^{i}\Delta k_j g^n_j \sum_{\ell = 0}^j\Delta k_\ell  g^n_\ell k_j k_\ell
        \\
        -2 \sum_{j=0}^{i}\Delta k_j(g^n_j)^2 k_j^2 - 2\sum_{j=\zeta}^{i}\Delta k_j (g^n_j)^2 k_j^2 \Bigg), 
    \end{aligned}
\end{equation}
with $\beta = i+j$, $\sigma = i-j$ and $\zeta = \lfloor \frac{i}{2} \rfloor$ above. 

Then, the approximate collision-flux at $t^n$ and $c_i$ is 
\begin{equation}\label{eqn:collision_flux}
    \partial_k \mathfrak Q[g](t^n,k_i) \approx \frac{1}{\Delta k_i}\Big( \mathfrak Q_{i + 1/2} - \mathfrak Q_{i - 1/2}\Big)g^n.
\end{equation}

 It was seen in \cite{waltontranFVS} that the forward-cascade collision-flux is quite stiff and can lead to very restrictive time steps if one uses an explicit time integration method.  Unfortunately, it would appear that the collision-flux defined by \eqref{eqn:collision_flux} exhibits even stiffer behavior.  For this reason, we will use implicit time integration methods.  

However, the collision operator is expensive to evaluate and coupling this cost with a nonlinear solve to evolve in time can negate any advantages one might hope to obtain from the increased time step size with an implicit integration scheme.  To mitigate this, we will simplify the evaluation of the flux term. {Then, let $$\delta_i\mathfrak Q = \mathfrak Q_{i+1/2} - \mathfrak Q_{i-1/2}, $$ denote the difference operator at $i$.  Using the definitions for $\mathfrak Q_{i+1/2}$ and $\mathfrak Q_{i-1/2}$ and explicitly computing their difference $\mathfrak Q_{i+1/2} - \mathfrak Q_{i-1/2}$, one arrives at a further simplification
\begin{equation}\label{simplified_flux}
    \begin{aligned}
        \delta_i\mathfrak Q g^n = 2\Bigg( \sum^M_{j=0}\Delta k_j g^n_j\Big\{g^n_{i+j}\Delta k_{i+j}(k_j - k_{i+j})^2 - g^n_{i-1}\Delta k_{i-1}(k_j - k_{i-1})^2 \Big\}\\
        + \sum^{i-1}_{j=0} g^n_j\Delta k_j g_{i-j-1}^n\Delta k_{i-j-1}(k_j + k_{i-j-1})^2 - g^n_{i}\Delta k_i \sum^i_{j=0}g^n_{j}\Delta k_j (k_i + k_j)^2 \\
        - 4\Big\{g^n_i\Delta k_i \sum_{j=0}^M g^n_j \Delta k_j k_i k_j - g^n_{i-1}\Delta k_{i-1}\sum^{i-1}_{j=0}g_j^n\Delta k_j k_{i-1}k_j \Big\} \\
        - 4 (g^n_i)^2k_i^2\Delta k_i^2 + 2 (g^n_{\lfloor \frac{i-1}{2}\rfloor})^2k^2_{\lfloor \frac{i-1}{2}\rfloor}\Delta k_{\lfloor \frac{i-1}{2}\rfloor}^2\Bigg).
    \end{aligned}
\end{equation}}

Thus, using the backward-Euler method would cast the fully discrete equation as
\begin{equation}\label{fully_discrete}
    \begin{aligned}
         g^{n+1}_i = g^n_i + \frac{\Delta t}{\Delta k_i}\delta_i\mathfrak Qg^{n+1}.
    \end{aligned}
\end{equation}

The above scheme is unconditionally stable and positivity preserving \cite{Higueras}, but first order in time.  In practice, we use the TR-BDF2 method (see section \ref{sec::num_tests} below) with an adaptive time-stepping procedure.  The adaptive step selection can be further constrained with a condition to guaranteee positivity \cite{diffeqdotjl}.  In the tests we perform, for the time intervals reported, adding such a constraint was not necessary and the solutions maintained positivity.  In an upcoming work we develop high-order implicit schemes tailored to the kinetic wave equation which guarantee positivity of the solutions.

In the next section, various initial distributions for the energy density will be used to test the above scheme.  Our main concern is testing the long-time behavior for which an implicit time integration method is essential. The total energy of the computed solutions is compared with the theoretical bound in equation \eqref{Decomposition0} of the introduction.  In addition to the total energy, the zeroth moment of the energy density, we will also consider higher moments.  The total energy in the interval $[0,L]$ is given by
\begin{equation}
    E_L(t) = \mathcal M^0_L[g](t) = \int_{0}^L g(t,k)k^2\mathrm d k.
\end{equation}
Thus, this quantity, and the higher moments, can be computed from the solution of \eqref{fully_discrete} by
\begin{equation}\label{eqn:moments_definition}
    \mathcal M^r_L[g](t^n) = \sum^M_{j=0}\Delta k_j g^n_j k_j^{r+1},
\end{equation}
recalling that the result of equation \eqref{fully_discrete} is $\tilde g = kg$, but we dropped the tilde notation in the derivation for convenience.

\section{Numerical Results}\label{sec::num_tests}
{In the numerical examples provided, the aim is to test the long-time behavior of the solutions with various initial conditions, both smooth and non-smooth.  Our choice of time integrator coupled with the midpoint rule in the collision integral make the scheme 2nd order accurate in frequency and time.  The code was implemented using MATLAB \cite{MATLAB}.  
    In the tests which follow, we use the Trapezoidal Backward-Differentiation (TR-BDF2) method (see \cite{HOSEA199621, trbdf2} and \cite{EDWARDS20111198} for example) which is a single-step, second order L-stable method. The TR-BDF2 method can be cast as an embedded Diagonally Implicit Runge-Kutta (DIRK) method \cite{hairer_wanner1993} with Butcher table \cite{HOSEA199621} given in Table \ref{tab::TRBDF2},
\begin{table}[h]
    \centering
    \begin{minipage}[t]{0.3\textwidth}
        \renewcommand*{\arraystretch}{1.25}
        \begin{tabular}{c|lll}
            $0$ & $0$ & $0$ & $0$ \\
            $\gamma$ & $d$ & $d$ & $0$ \\
            $1$ & $w$ & $w$ & $d$ \\
            \hline
            \text{ } & $w$ & $w$ & $d$\\
            \hline
            \text{ } & $\frac{1-w}3$ & $\frac{1+3w}3$ & $\frac{d}{3}$
        \end{tabular}
    \end{minipage} 
\caption{Butcher table for TR-BDF2 method cast as an embedded Diagonally Implicit Runge-Kutta (DIRK) method.}
\label{tab::TRBDF2}
\end{table}
where $\gamma  = 2-\sqrt{2}$, $d = \frac{\gamma}{2} $ and $w = \frac{\sqrt{2}}{4}$. The embedded formulation provides a built-in means to estimate the error of solutions and adaptively adjust the time step.  This is precisely the ode23tb integrator in MATLAB.  The adaptive time step allows us to reach quite large values of $T$ quickly. } 

In all tests, we compare the decay rate of the total energy with the theoretical bound $\mathcal O(\frac{1}{\sqrt{t}})$ (see equation \eqref{Decomposition0} in the introduction) and provide a best fit line with computed slope given in the legend for each of these decay estimates.  
  
{To test the total energy decay, the simulations are run to $T = 1e+04$ and the truncation parameter, $L$, may also vary across tests.  In every test, $\Delta k_i = \Delta k = 0.5$ for simplicity.  As the theory suggests the decay bound should hold for { \it any } finite interval. For each test, we provide simulations for two values of the truncation parameter.  The total energy decay for each truncations parameter is compared.}

\subsection*{Test 1}
The first example employs an initial distribution which is compactly supported at lower frequencies and given by 
\begin{equation}\label{eqn::mollifier_ic}
    g(0,k) = \begin{cases}
             \exp\Big(\frac{1}{10(|k-15|^2-1)}\Big) & \text{if }  |k-15| \leq 1\\
            0 & \mathrm{ otherwise } ,
\end{cases}
\end{equation}
as shown in figure \ref{fig::mollifier_ic}. Here, as we are interested in the long-time behavior in any finite interval of the wave-number domain, the truncation parameter is set to {$L=30$ and $L=100$}.  Larger values of the truncation parameter are used in subsequent tests.  

In Figure \ref{fig:total_energy_mollifier_ic}, the total energy for {both values of $L$ is plotted against the theoretical rate. The plot indicates that the total energy in the finite intervals considered} is conserved for a short time $t\in[0, t_{\textrm{Cons}}^1]$, as shown in \cite{soffer2019energy}.  {Note that the energy vs time plots show the steeper
$(t^{-0.764}-t^{-0.729}) $ descent than the right-hand side of \eqref{Decomposition0} predicts $(t^{-1/2})$ since $\mathcal O (t^{-0.764})$, $ \mathcal O   (t^{-0.729})$  $\le \mathcal O  (t^{-1/2})$. Let us remark that in \eqref{Decomposition0}, the authors of \cite{soffer2019energy} only obtained an upper bound on the decay rates since  a lower bound cannot be expected. We will see that this is true for the remaining tests as well.}   For $t>t^1_{\textrm{Cons}}$, the total energy in the interval begins to decay rapidly.  Interestingly, it would appear that the total energy is again conserved for a short time before decaying again at a slower rate that more closely matches the theoretical bound.  Similar behavior is observed for the other initial conditions considered.  We provide a best fit line to the total energy data, which shows that the long-time decay rate for the total energy is indeed bounded by the theoretical rate as the reported slope is less than that of the theoretical rate bound. {In Figure \ref{fig:total_energy_mollifier_ic}, the shelf-type behavior of the energy plot may come from the fact that the solutions blow-up in a self-similar manner. However, this is a guess and we do not have any proof for that.  As proved in \cite{soffer2019energy}, the energy of the solutions will move out of any finite interval $[0,L]$. As a results, it is expected that the moments of any orders are also decaying in time as well.  Figure \ref{fig:moments_mollifier_ic} captures this theoretical predictions for several moments of the solution. }
\begin{figure}
    \centering
    \includegraphics[width=0.75\textwidth]{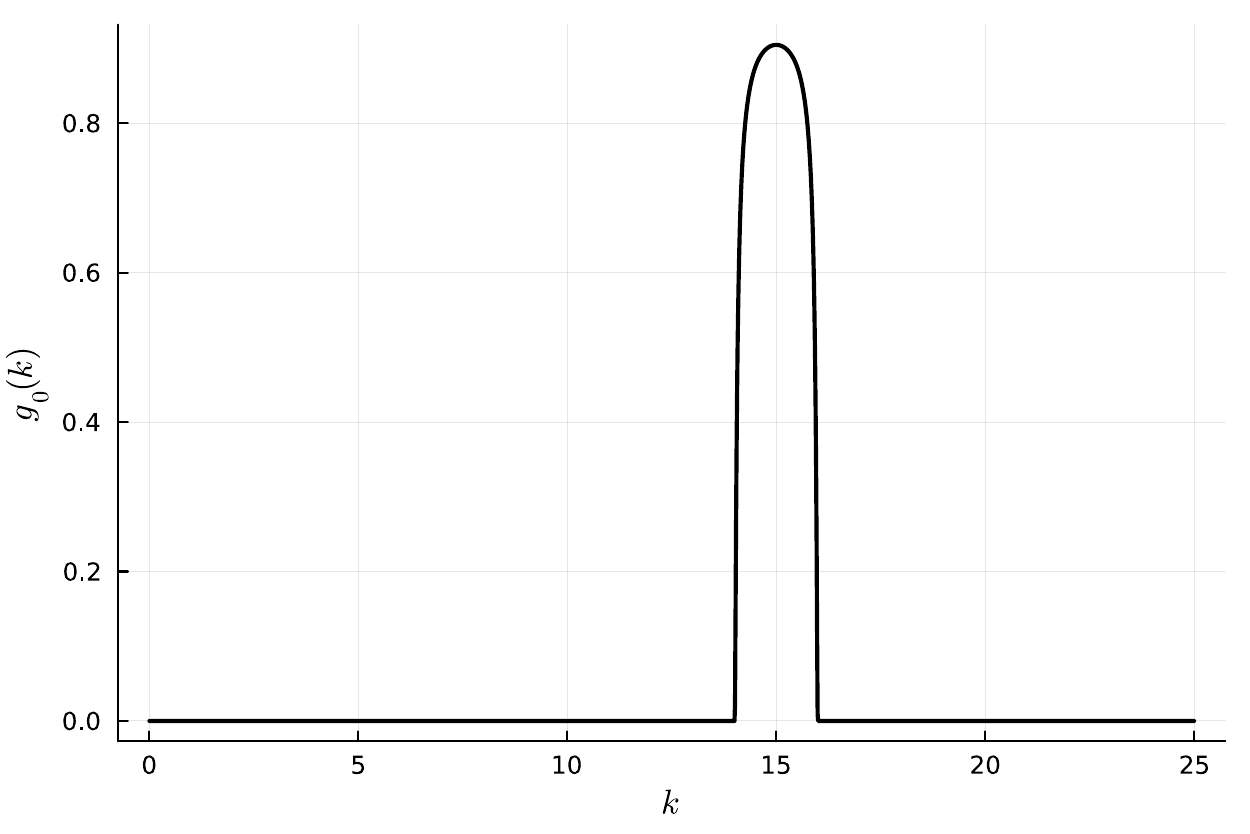}
    \caption{Initial condition corresponding to equation \eqref{eqn::mollifier_ic}. }
    \label{fig::mollifier_ic}
\end{figure}

\begin{figure}
    \centering    \includegraphics[width=0.75\textwidth]{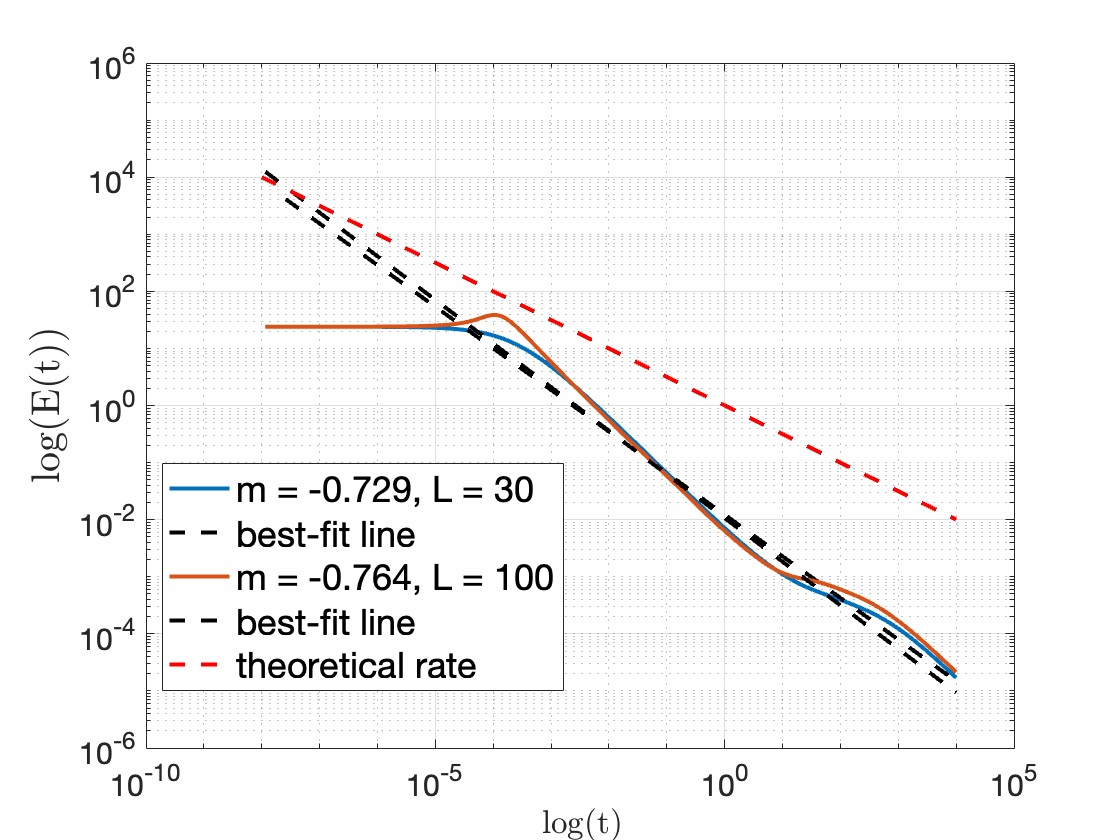}
    \caption{Decay rate of total energy corresponding to \eqref{eqn::mollifier_ic} with theoretical rate, $t^{-\frac12}$, shown for reference.}
    \label{fig:total_energy_mollifier_ic}
\end{figure}

\begin{figure}
    \centering
    \begin{subfigure}{0.49\textwidth}
        \centering
        \includegraphics[width=\textwidth]{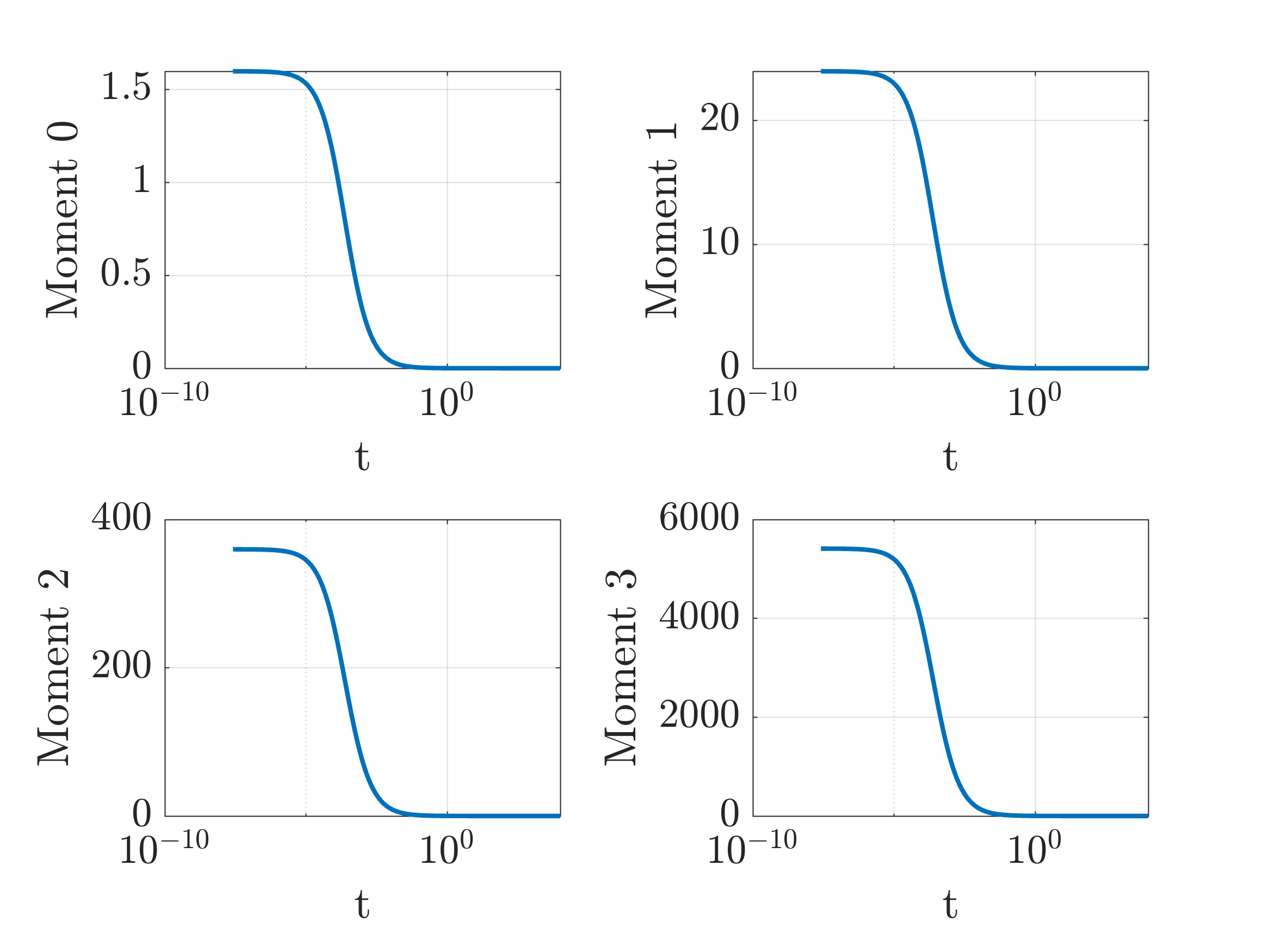}
        \caption{$L = 30$}
    \end{subfigure}
    \hfill
    \begin{subfigure}{0.49\textwidth}
        \centering
        \includegraphics[width=\textwidth]{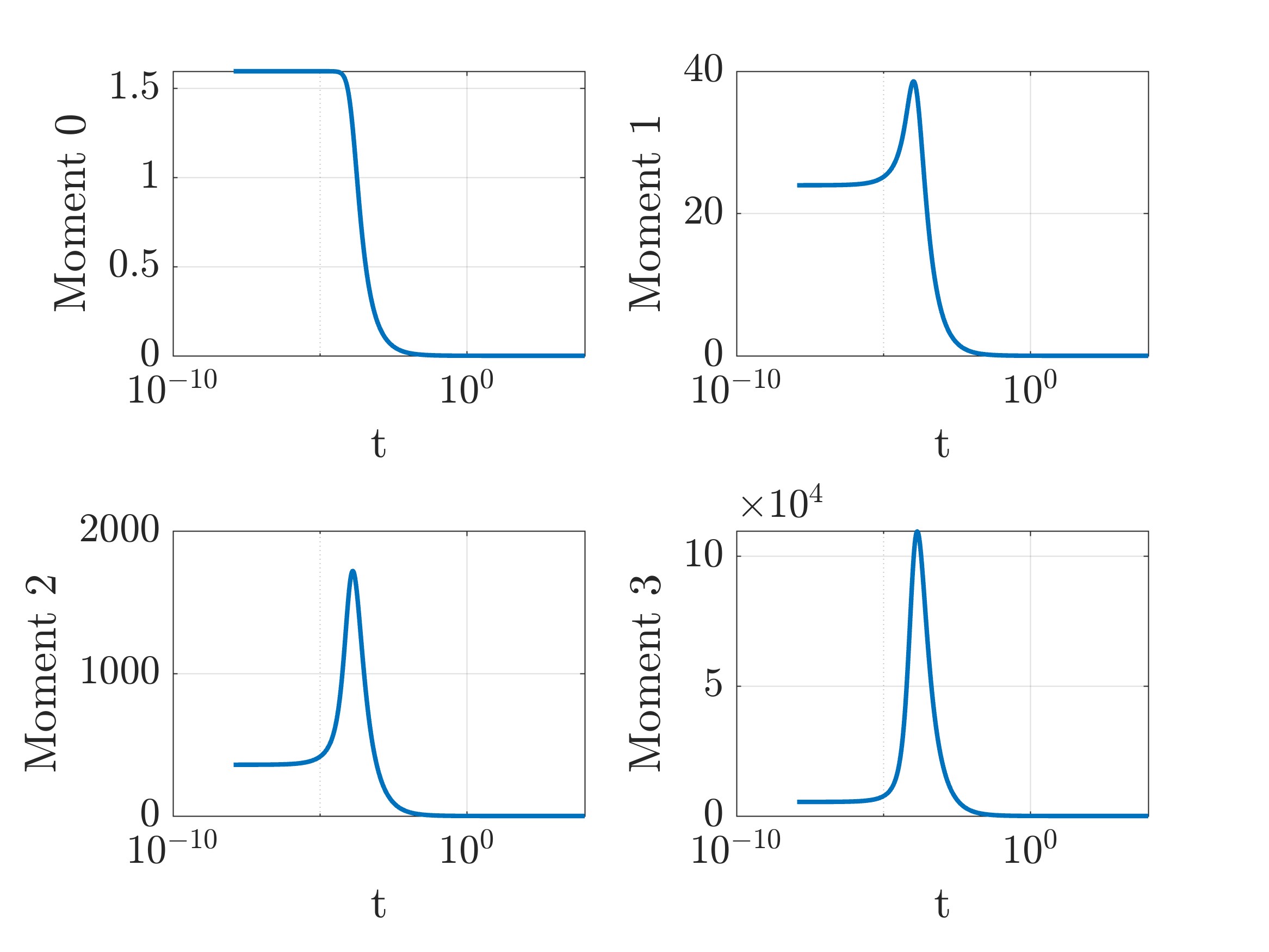}
        \caption{$L = 100$}
    \end{subfigure}
    \caption{Moments of the energy density for initial condition \eqref{eqn::mollifier_ic} with (a) $L=30$ and (b) $L=100$.}
    \label{fig:moments_mollifier_ic}
\end{figure}

\subsection*{Test 2}
For our second example, we initialize the solver with a discontinuous distribution that has support ranging from low wavenumbers to higher wavenumbers as depicted in Figure \ref{fig:disc_line_ic} and given in equation \eqref{eqn::disc_line_ic} below,
\begin{equation}\label{eqn::disc_line_ic}
    g(0,k) = \begin{cases}
    1 - \frac{1}{130}(k-20) & \text{if } 20 \leq k \leq 150 \\
    0 & \mathrm{ otherwise } .
\end{cases}
\end{equation}
For this test, {the truncation parameter are set to $L=200$ and $L=300$. The simulation is run until $T=1e+04$}. In this and the next test, the initial energy density profile is spread throughout the computational domain, increasing the total energy in the interval we consider. 
\begin{figure}
    \centering
    \includegraphics[width=0.75\textwidth]{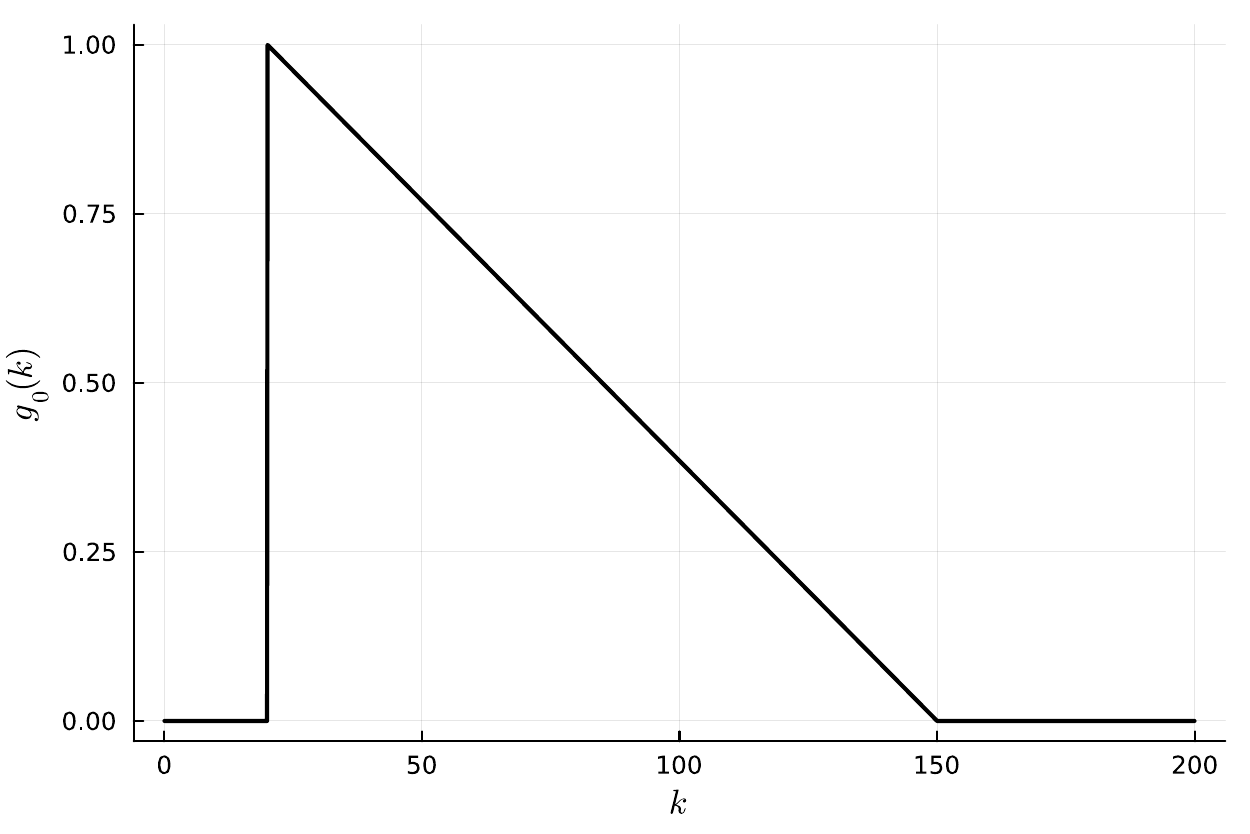}
    \caption{Initial condition corresponding to equation \eqref{eqn::disc_line_ic}.}
    \label{fig:disc_line_ic}
\end{figure}

\begin{figure}
    \centering
    \includegraphics[width=0.75\textwidth]{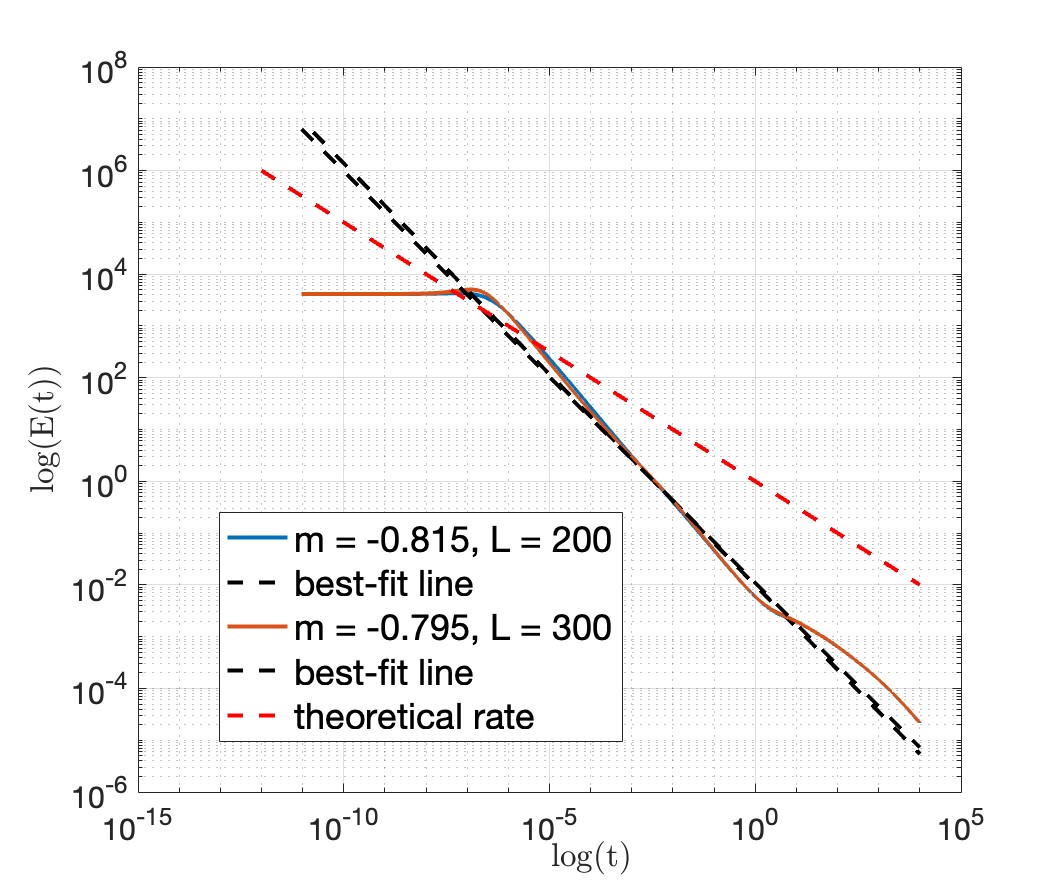}
    \caption{Decay rate of total energy corresponding to \eqref{eqn::disc_line_ic} with theoretical rate and best fit line through the data shown for reference.}
    \label{fig:total_energy_disc_line_ic}
\end{figure}

\begin{figure}
    \centering
    \begin{subfigure}{0.49\textwidth}
        \centering
        \includegraphics[width=\textwidth]{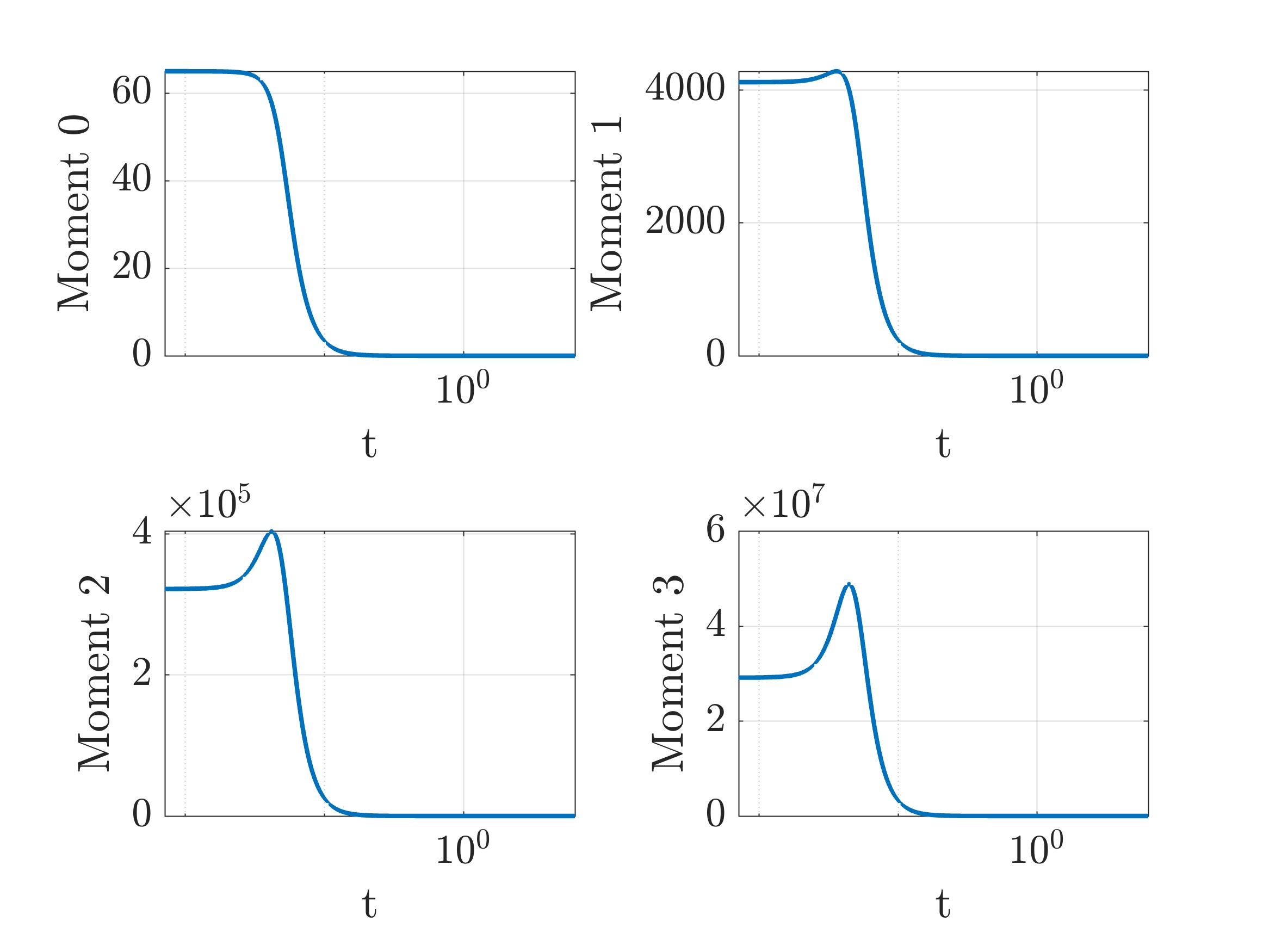}
        \caption{$L = 200$}
    \end{subfigure}
    \hfill
    \begin{subfigure}{0.49\textwidth}
        \centering
        \includegraphics[width=\textwidth]{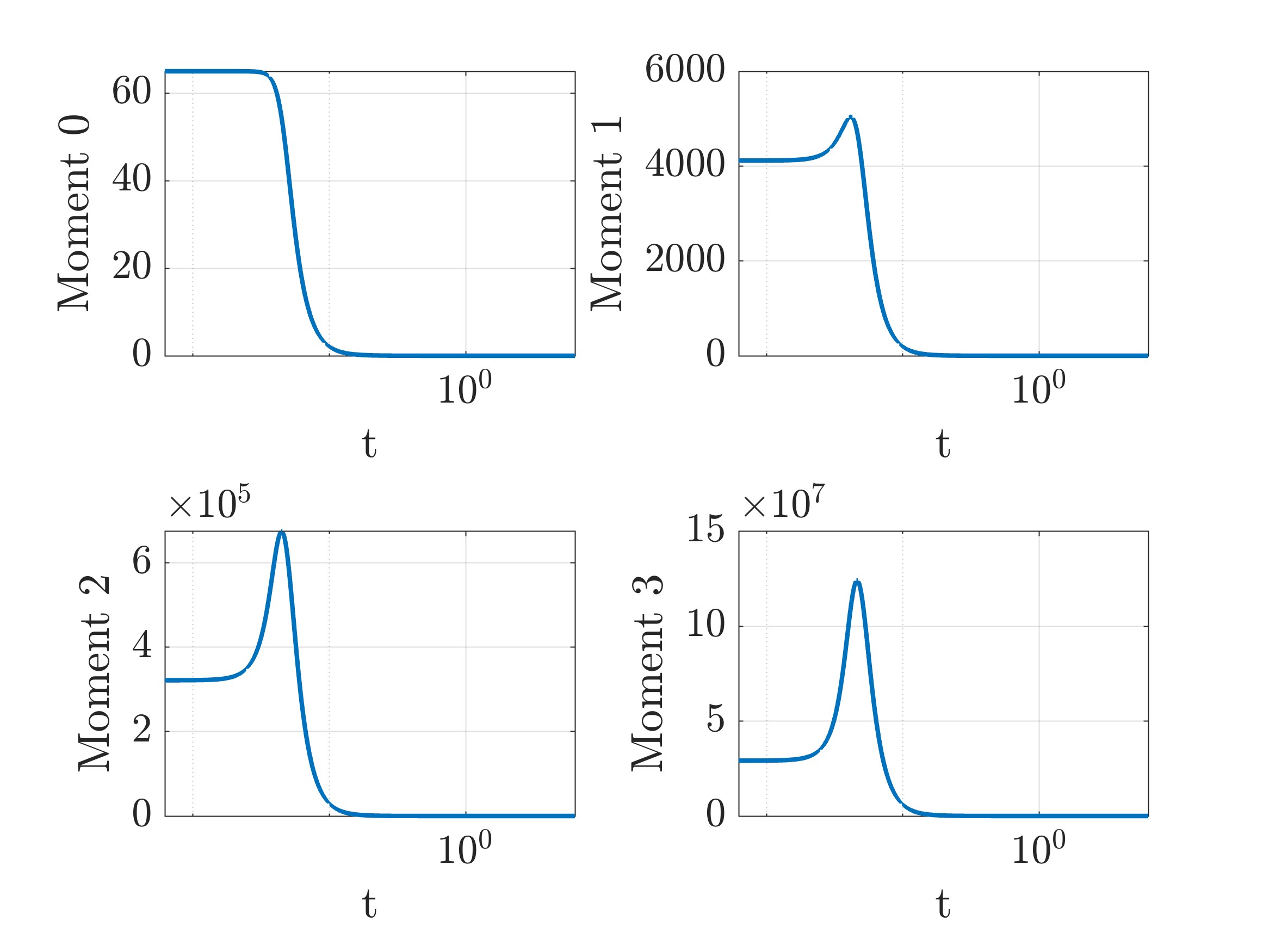}
        \caption{$L = 300$}
    \end{subfigure}
    \caption{Moments of the energy density for initial condition \eqref{eqn::disc_line_ic} with (a) $L=200$ and (b) $L=300$.}
    \label{fig:moments_disc_line_ic}
\end{figure}


\subsection*{Test 3}
Our last test is depicted in Figure \ref{fig:triple_bump_ic} and given by
\begin{equation}
\label{eqn::triple_bump_ic}
    g(0,k) = \frac{1}{10}\Bigg(\exp\Big(-\frac{(k-50)^2}{100}\Big) + \frac{1}{2}\exp\Big(-\frac{(k-75)^2}{100}\Big) + \exp\Big(-\frac{(k-100)^2}{100}\Big) \Bigg),
\end{equation}
for which we employ the same {value of the truncation parameters as in the preceding test, $L=200,300$, and run until $T=1e+04$ once more.  The total energy is shown in figure \ref{fig:total_energy_triple_bump_ic}.  Again, we see that there is good agreement with the theoretical bound and the computed decay rate for both $L$ values. }
\begin{figure}
    \centering
    \includegraphics[width=0.75\textwidth]{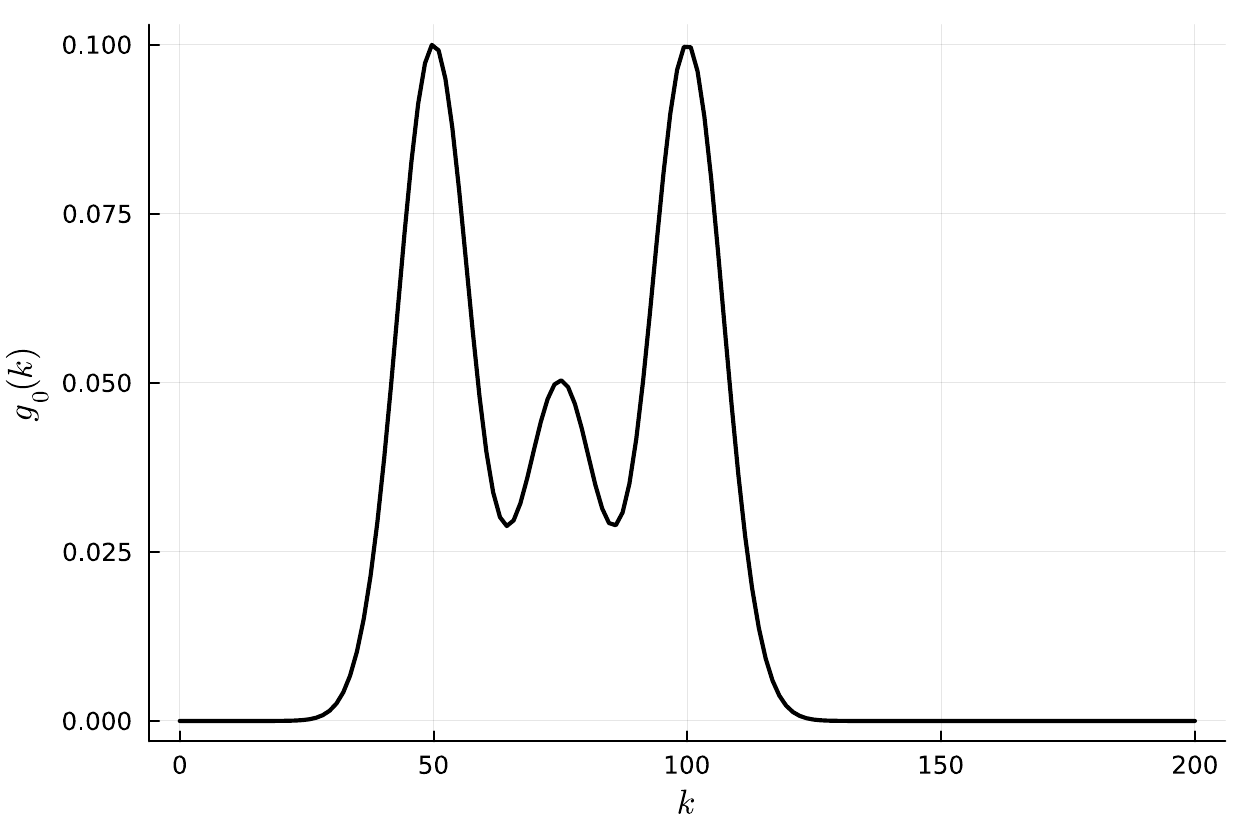}
    \caption{Initial condition corresponding to equation \eqref{eqn::triple_bump_ic}.}
    \label{fig:triple_bump_ic}
\end{figure}

\begin{figure}
    \centering
    \includegraphics[width=0.75\textwidth]{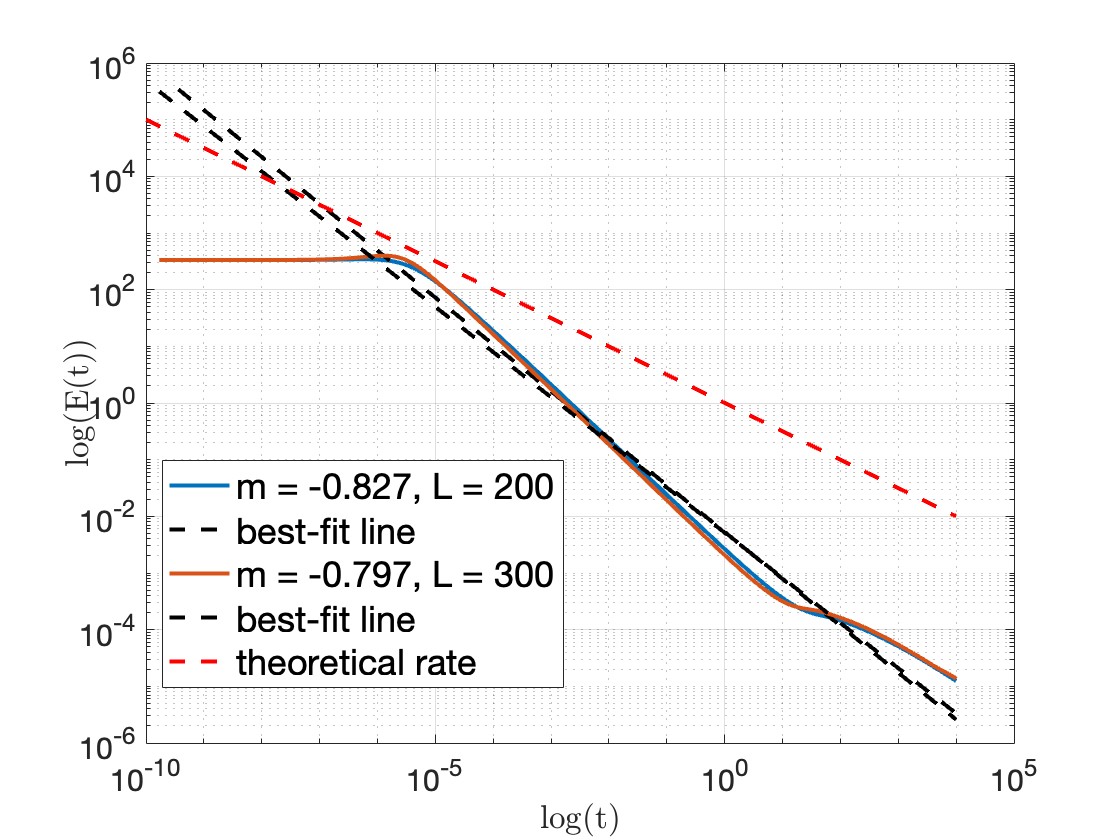}
    \caption{Decay rate of total energy corresponding to \eqref{eqn::triple_bump_ic} with theoretical rate and best fit line through the data shown for reference.}
    \label{fig:total_energy_triple_bump_ic}
\end{figure}

\begin{figure}
    \centering
    \begin{subfigure}{0.49\textwidth}
        \centering
        \includegraphics[width=\textwidth]{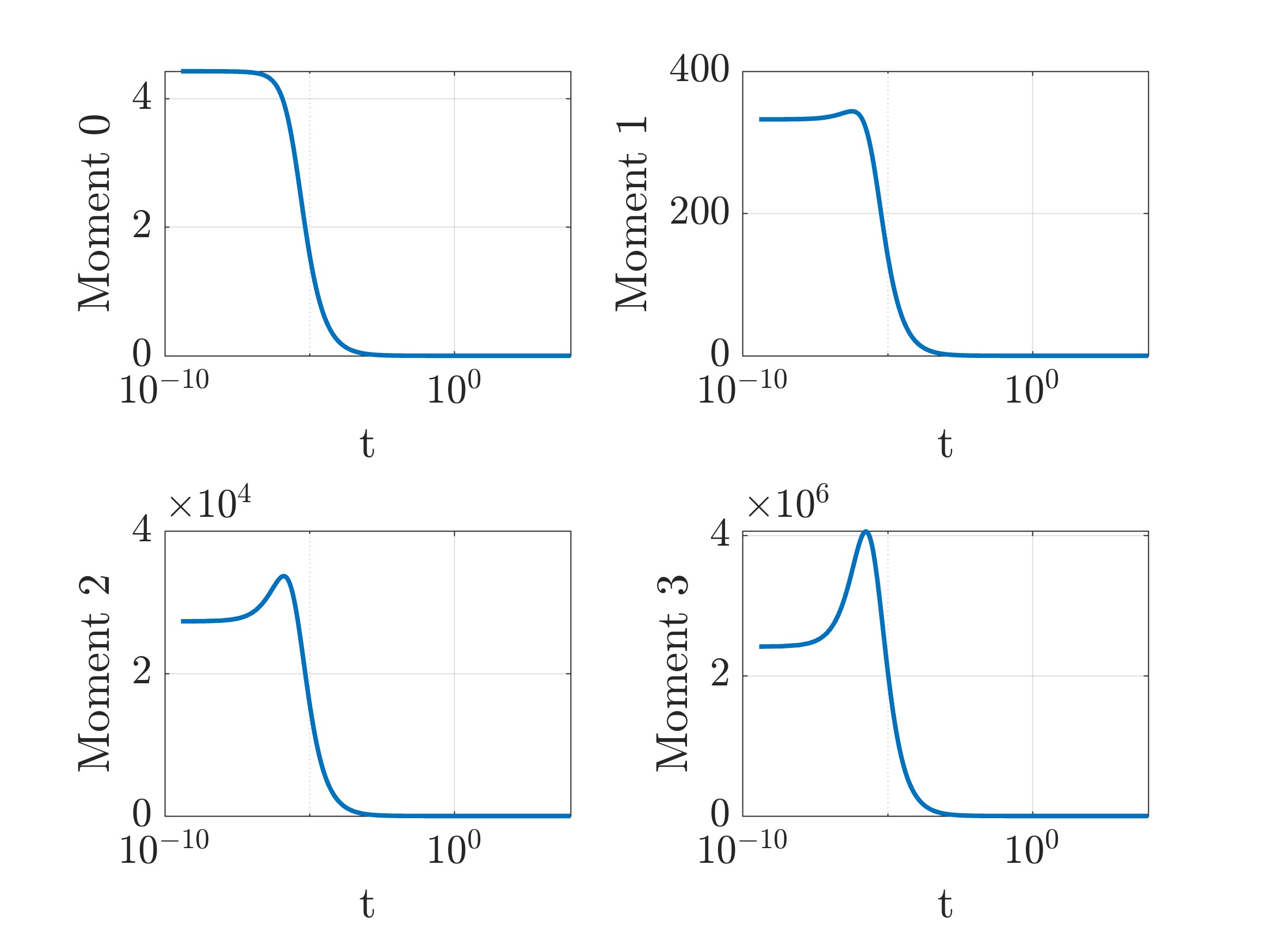}
        \caption{$L = 200$}
    \end{subfigure}
    \hfill
    \begin{subfigure}{0.49\textwidth}
        \centering
        \includegraphics[width=\textwidth]{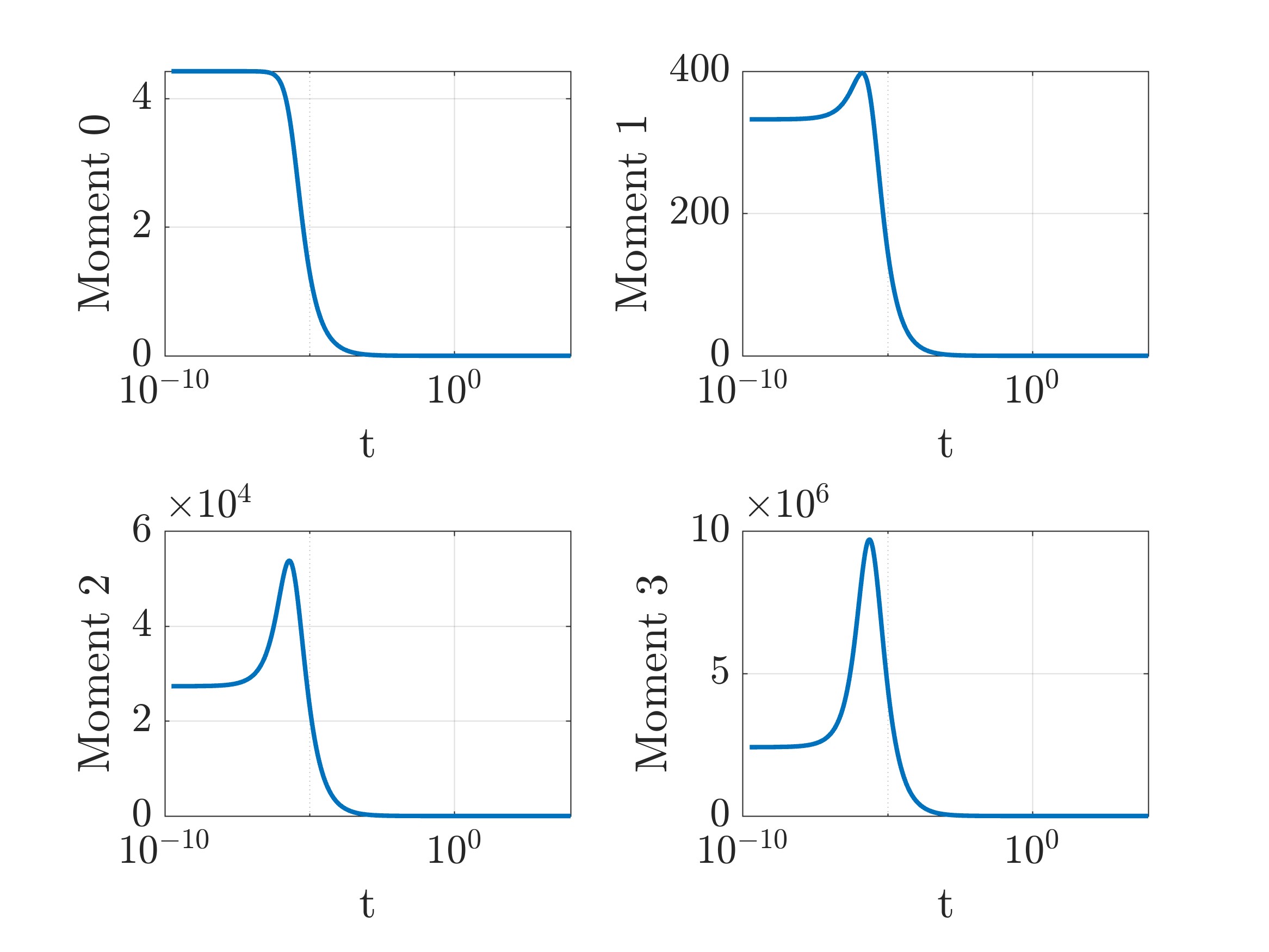}
        \caption{$L = 300$}
    \end{subfigure}
    \caption{Moments of the energy density for initial condition \eqref{eqn::triple_bump_ic} with (a) $L=200$ and (b) $L=300$.}
    \label{fig:moments_triple_bump_ic}
\end{figure}
{Further, we provide a plot in figure \ref{fig::decay_with_inset} of the total energy decay for $L=300$ with inset plots of $g(t,k)$ at a few snapshots in time.  The times chosen are at the loss of conservation, or onset of decay, and the beginning and end of the short interval where the energy again appears to be conserved.  This behavior is present in all of the preceding test cases as well. }
\begin{figure}
    \centering
    \includegraphics[width=1.\linewidth, height=.58\linewidth]{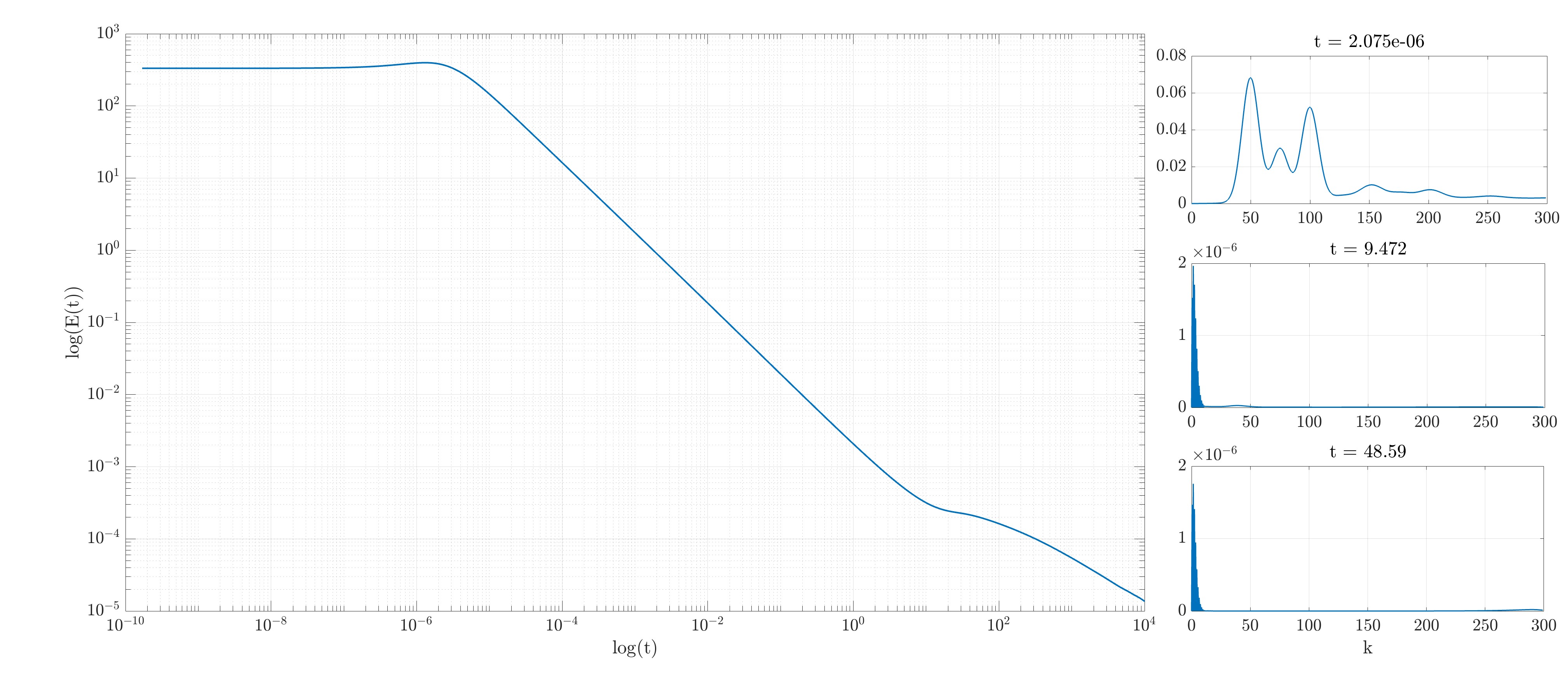}
    \caption{Total energy decay for initial condition \eqref{eqn::triple_bump_ic} with $L=300$.  The subfigures on the right show snapshots in time of the energy density $g(t,k)$.  The top figure shows the energy density just before the onset of decay, while the middle and bottom figures show the beginning and end of the plateau region that can also be observed for every other initial condition, though at different times.}
    \label{fig::decay_with_inset}
\end{figure}
{For this test, we provide an experimental convergence test.  We use a reference solution with $\Delta k = \frac{1}{20}$ with $L=200$.  Here we stop the calculation at $T=10$.  The grid is halved for each successive computation. Specifically, we set $\Delta k = \frac{1}{2}, \frac{1}{4}, \frac{1}{8},\frac{1}{16}$.  These coarse grid solutions are interpolated onto the fine grid.  In figure \ref{fig::triple_bump_convergence}, we show the relative $L^1$ and $L^\infty$ error with respect to the fine grid solution.  The evolution corresponding to initial condition \eqref{eqn::triple_bump_ic} is smooth and we obtain the claimed second order convergence.
\begin{figure}
    \centering
    \includegraphics[width=\linewidth]{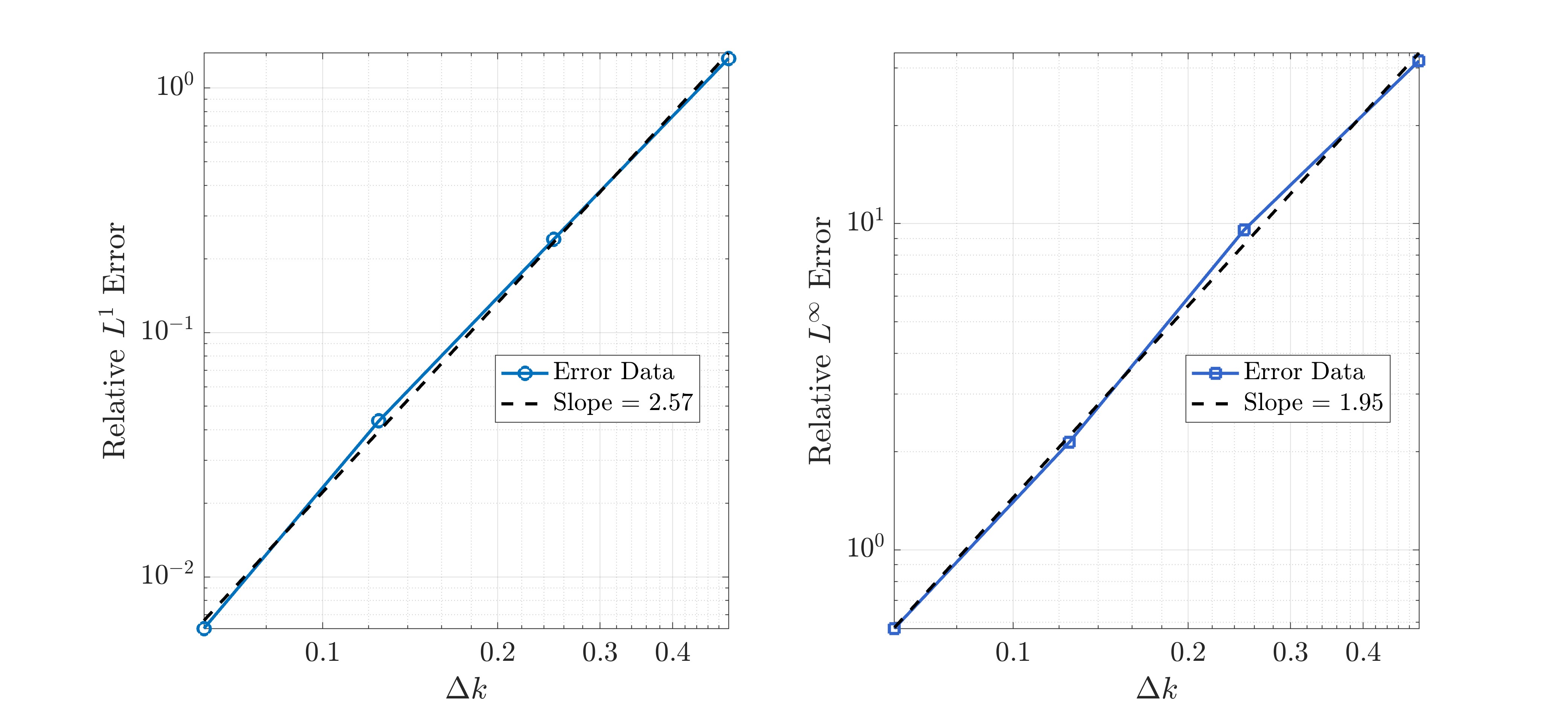}
    \caption{Convergence plots of using the relative $L^1$ and $L^\infty$ norm with respect to a fine grid solution.  Here, the simulation is carried out to $T=10$ with $L=200$.  We see that 2nd order convergence is obtained.}
    \label{fig::triple_bump_convergence}
\end{figure}}


\section{Conclusions}

In this article, we were able to extend the work found in \cite{waltontranFVS} from the simple case when only  the forward-cascade part of the collision operator is kept  to much more complicated situation when the complete collision operator is treated.  A conservative form for the energy density was derived and subsequently an implicit finite volume scheme was given.  The implicit finite volume scheme does not suffer from a restrictive stability condition on the time step as is true for the explicit scheme for the forward-cascade equation. This was achievable {due to a further} simplification at the discrete level of the collision-flux.
A future work will be dedicated to a 3-WKE tailored time-step control for higher-order time integration to ensure positivity for a given positive initial distribution.
In each numerical test, we verified that the energy was conserved locally in time on the finite intervals we considered and that the theoretical bound \eqref{Decomposition0} was adhered to with similar computed rates across tests.  This is in strong agreement with the results found in \cite{soffer2019energy}.  In this work, we restrict our study to the long-time asymptotic behavior of the $3$-wave kinetic equation. Different types of $n$-wave kinetic equations exhibit distinct long-time dynamics. The numerical study of these equations is an important line of research that has not yet been fully explored. Among recent developments, in the context of 4-wave kinetic equations, \cite{qi2025fast} presents a fast Fourier spectral approach allowing for efficient and accurate simulations.



 \bibliographystyle{elsarticle-num} 
 \bibliography{references.bib}





\end{document}